\documentclass{amsart}
\usepackage{amsmath}
\usepackage{amsthm}
\usepackage{amssymb}
\usepackage[margin=0.5in]{caption}
\usepackage{dsfont}
\usepackage{enumitem}
\usepackage{graphicx}
\usepackage[hidelinks]{hyperref}
\usepackage[capitalise]{cleveref}
\usepackage{mathrsfs}
\usepackage{mathtools}
\usepackage{nameref}
\usepackage[new]{old-arrows}
\usepackage{stmaryrd}
\usepackage{thmtools}
\usepackage{url}
\usepackage{xcolor}
\usepackage{xparse}

\usepackage{tikz}
\usetikzlibrary{cd}
\usetikzlibrary{positioning, angles, quotes}
\newcommand{\nospacepunct}[1]{\makebox[0pt][l]{\,#1}} 

\theoremstyle{plain}
\newtheorem{thm}{Theorem}[section]
\newtheorem*{thm*}{Theorem}

\newtheorem{cor}[thm]{Corollary}
\newtheorem*{cor*}{Corollary}

\newtheorem{prop}[thm]{Proposition}
\newtheorem*{prop*}{Proposition}

\newtheorem{lem}[thm]{Lemma}
\newtheorem*{lem*}{Lemma}

\newtheorem*{exer*}{Exercise}

\newtheorem{conj}[thm]{Conjecture}
\newtheorem*{conj*}{Conjecture}

\theoremstyle{definition}
\newtheorem{const}[thm]{Construction}
\newtheorem*{const*}{Construction}

\newtheorem{defn}[thm]{Definition}
\newtheorem*{defn*}{Definition}

\newtheorem{ex}[thm]{Example}
\newtheorem*{ex*}{Example}

\newtheorem{q}[thm]{Question}
\newtheorem*{q*}{Question}

\theoremstyle{remark}
\newtheorem{rem}[thm]{Remark}
\newtheorem*{rem*}{Remark}

\theoremstyle{plain}

\Crefname{thm}{Theorem}{Theorems}
\Crefname{defn}{Definition}{Definitions}
\Crefname{lem}{Lemma}{Lemmata}
\Crefname{cor}{Corollary}{Corollaries}

\newcommand{\BS}{\mathrm{BS}}

\newcommand{\SL}{\mathrm{SL}}

\newcommand{\F}{\mathtt{F}}
\newcommand{\FP}{\mathtt{FP}}

\newcommand{\C}{\mathbb{C}}
\newcommand{\N}{\mathbb{N}}
\newcommand{\Q}{\mathbb{Q}}
\newcommand{\R}{\mathbb{R}}
\newcommand{\Z}{\mathbb{Z}}

\newcommand{\inv}{^{-1}}

\DeclareMathOperator{\cd}{cd}

\DeclareMathOperator{\Div}{Div}

\DeclareMathOperator{\Ext}{Ext}

\DeclareMathOperator{\id}{id}

\DeclareMathOperator{\Ore}{Ore}
\DeclareMathOperator{\rk}{rk}

\DeclareMathOperator{\Tor}{Tor}

\newcommand{\vertii}[1]{{\left\vert\kern-0.25ex\left\vert #1 \right\vert\kern-0.25ex\right\vert}}
\newcommand{\vertiii}[1]{{\left\vert\kern-0.25ex\left\vert\kern-0.25ex\left\vert #1 \right\vert\kern-0.25ex\right\vert\kern-0.25ex\right\vert}}

\makeatletter
\newsavebox{\@brx}
\newcommand{\llangle}[1][]{\savebox{\@brx}{\(\m@th{#1\langle}\)}%
  \mathopen{\copy\@brx\mkern2mu\kern-0.9\wd\@brx\usebox{\@brx}}}
\newcommand{\rrangle}[1][]{\savebox{\@brx}{\(\m@th{#1\rangle}\)}%
  \mathclose{\copy\@brx\mkern2mu\kern-0.9\wd\@brx\usebox{\@brx}}}
\makeatother

\newcounter{comments}

\title[Division rings for virtually compact special and $3$-manifold groups]{Division rings for group algebras of virtually compact special groups and $3$-manifold groups}
\author{Sam P.~ Fisher}
\address[S.~P.~Fisher]{University of Oxford, Oxford, OX2 6GG, UK}
\email{sam.fisher@maths.ox.ac.uk}

\author{Pablo S\'anchez-Peralta}
\address[P.~ S\'anchez-Peralta]{Universidad Aut\'onoma de Madrid, Madrid, Spain}
\email{pablo.sanchezperalta@uam.es}

\begin{document}

\begin{abstract}
    Let $k$ be a division ring and let $G$ be either a torsion-free virtually compact special group or a finitely generated torsion-free $3$-manifold group. We embed the group algebra $kG$ in a division ring and prove that the embedding is Hughes-free whenever $G$ is locally indicable. In particular, we prove that Kaplansky's Zero Divisor Conjecture holds for all group algebras of torsion-free $3$-manifold groups. The embedding is also used to confirm a conjecture of Kielak and Linton. Thanks to the work of Jaikin-Zapirain and Linton, another consequence of the embedding is that $kG$ is coherent whenever $G$ is a virtually compact special one-relator group.
    
    If $G$ is a torsion-free one-relator group, let $\overline{kG}$ be the division ring containing $kG$ constructed by Lewin and Lewin. We prove that $\overline{kG}$ is Hughes-free whenever a Hughes-free $kG$-division ring exists. This is always the case when $k$ is of characteristic zero; in positive characteristic, our previous result implies that this happens when $G$ is virtually compact special.
\end{abstract}

\maketitle

\section{Introduction}

The Kaplansky Zero Divisor Conjecture asserts that $kG$ is a domain, where $k$ is a field and $G$ is a torsion-free group. One of the first approaches to the Zero Divisor Conjecture was to embed $kG$ in a division ring; this strategy was successfully implemented by Mal'cev and Neumann, independently, in the case where $G$ is bi-orderable \cite{Malcev_series, Neumann_series}, where the group algebra is embedded into the Mal'cev--Neumann division ring of power series with well-ordered supports. Note that while the Kaplansky Zero Divisor Conjecture is still wide open, there are also no counterexamples to the following, a priori stronger, conjecture.

\begin{conj}\label{conj:divRing}
    If $G$ is a torsion-free group and $k$ is a field, then $kG$ embeds into a division ring.
\end{conj}

When trying to prove the Zero Divisor Conjecture in characteristic zero, there are many analytic tools at our disposal. Most notably, one can use the well-developed arsenal of $L^2$-techniques to prove the Strong Atiyah Conjecture for a torsion-free group $G$, which Linnell showed to be equivalent to the statement that the division closure of $\C G$ in $\mathcal U(G)$ is a division ring, where $\mathcal U(G)$ denotes the algebra of operators affiliated to the $G$-equivariant bounded operators on $L^2(G)$ (\cite{LinnellDivRings93}, and see \cite[Section 10]{Luck02} for relevant background). Thus, the Strong Atiyah Conjecture for a torsion-free group $G$ implies \cref{conj:divRing} when $k = \C$. The Strong Atiyah Conjecture is known for many notable classes of groups, including
\begin{enumerate}
    \item torsion-free groups in Linnell's class $\mathcal C$ \cite[Theorem 1.5]{LinnellDivRings93}, defined as the smallest class of groups containing all free groups and closed under elementary amenable extensions and directed unions. Linnell's class $\mathcal C$ was recently shown to contain all $3$-manifold groups by Kielak and Linton \cite{KielakLinton_3mfldAtiyah};
    \item the class of locally indicable groups, due to Jaikin-Zapirain and L\'opez-\'Alvarez \cite{JaikinLopez_Atiyah};
    \item the class of virtually compact special groups, as defined by Haglund and Wise \cite{HaglundWise_special}, which are groups that are virtually the fundamental group of a nonpositively curved compact cube complex that avoids certain pathological hyperplane configurations. This is due to Schreve \cite{Schreve_AtiyahVCS}.
\end{enumerate}
Our list is far from exhaustive, and we refer the reader to Jaikin-Zapirain's survey \cite{Jaikin_l2survey} for a good account of what is known about the Strong Atiyah Conjecture.

The purpose of this article is to extend some of these embedding results beyond characteristic zero. Because of the lack of a suitable analogue of $\mathcal U(G)$ in positive characteristic, the methods we employ are necessarily more algebraic. Throughout, we will work with crossed products $k*G$ of a torsion-free group $G$ with a division ring $k$. The elements of $k*G$ are formal sums of elements of $G$ with coefficients in $k$, where the multiplication is twisted (see \cref{subsec:LinnellHughes} for more details). We emphasise that the group algebra $kG$ is an example of a crossed product, where the twisting is trivial. A central concept will be that of a Linnell embedding: if $G$ is a torsion-free group and $k$ is a division ring, we say that an embedding $\varphi \colon k*G \hookrightarrow \mathcal D$ is \textit{Linnell} if $\mathcal D$ is a division ring generated by the image of $k*G$ and it satisfies the following linear independence condition:
\begin{enumerate}
    \item[(L)] let $H \leqslant G$ be a subgroup and let $S$ is a system of right coset representatives for $H$ in $G$, then $S$ is left-linearly independent over the division closure of $k*H$ in $\mathcal D$.
\end{enumerate}
The concept will be recalled in more detail in \cref{sec:prelims}. In this case, $\mathcal D$ is called a \textit{Linnell division ring} for $k * G$. When $G$ is locally indicable, a Linnell division ring is unique up to $k*G$-isomorphism \cite{HughesDivRings1970} and called a \textit{Hughes-free division ring} for $k*G$; it is then denoted by $\mathcal D_{k*G}$. These notions will be recalled in more detail in \cref{sec:prelims}. We highlight the following conjecture, which is a strengthening of \cref{conj:divRing} and is due to Jaikin-Zapirain and Linton in the case of a group algebra $kG$ over a field $k$ \cite[Conjecture 1]{JaikinLinton_oneRelCoherence}.

\begin{conj}\label{conj:pAtiyah}
    Let $G$ be a torsion-free group and let $k$ be a division ring. Then any crossed product $k*G$ embeds in a Linnell division ring, which is unique up to $k*G$-isomorphism.
\end{conj}

We view \cref{conj:pAtiyah} as a version of the Strong Atiyah Conjecture in positive characteristic. Our main result is the following.

\begin{thm}\label{thm:main}
    Let $k$ be a division ring, let $G$ be a torsion-free group, and let $k*G$ be any crossed product. If $G$ is
    \begin{enumerate}[label=(\arabic*)]
        \item\label{item:vcs} (\cref{cor:VCSinField})  virtually the fundamental group of a compact special cube complex, then $k*G$ embeds into a Linnell division ring, which is unique up to $k*G$-isomorphism;
        \item\label{item:3} (\cref{thm:3mfld,thm:nonori3mfld}) the fundamental group of a compact $3$-manifold, then $k*G$ has an embedding into a division ring, and the embedding can be made Hughes-free when $G$ is locally indicable.
    \end{enumerate}
\end{thm}

Our proof of \ref{item:vcs} builds heavily on arguments of Linnell and Schick in \cite{LinnellSchick_AtiyahExt} and on Schreve's proof of the Strong Atiyah Conjecture for virtually compact special groups \cite{Schreve_AtiyahVCS}, where he introduced the \textit{factorisation property}, the key tool which allows us to produce the embeddings. The first main ingredient in the proof of \ref{item:3} is a result of Friedl--Schreve--Tillmann \cite[Theorem 3.3]{FriedlSchreveTillmann_ThurstonFox} which, together with \cite[Theorem 1.1]{KielakLinton_3mfldAtiyah}, allows us to show that the fundamental group of any irreducible $3$-manifold that is not a closed graph manifold has the factorization property and that the group algebras of their fundamental groups embed in division rings (see \cref{thm:nonGraphDivRings}). The second main ingredient is the graph of rings construction studied in \cref{sec:GraphsRings}, which is what is needed to cover the case where $M$ is a closed graph manifold; namely, we prove the following combination theorem which allows us to produce our embeddings.

\begin{thm}[\cref{cor:locIndGraphDivRing,cor:atiyahGraph}]
    Let $k$ be a division ring and let $\mathcal G_\Gamma = (G_v, G_e)$ be a graph with fundamental group $G$.
    \begin{enumerate}
        \item Fix a crossed product $k*G$ and suppose that $G_v$ is locally indicable and that there is a Hughes-free embedding $k*G_v \hookrightarrow \mathcal D_{k*G_v}$ for each vertex $v$. Then $k*G$ embeds into a division ring.
        \item Suppose that $k$ is a subfield of $\C$ and that every vertex group $G_v$ satisfies the Strong Atiyah Conjecture (over $k$). Then the group algebra $kG$ embeds into a division ring.
    \end{enumerate}
\end{thm}

The case where $G$ is the fundamental group of a non-orientable $3$-manifold is not significantly different and is handled in \cref{sec:appendix} in order to keep the exposition as straightforward as possible.

The following corollary of \cref{thm:main} is immediate. We do not need to assume that the $3$-manifold group is finitely generated since the Zero Divisor Conjecture can be verified locally.

\begin{cor}\label{cor:3mfldKap}
    If $G$ is torsion-free and the fundamental group of a $3$-manifold, then $k*G$ is a domain for any division ring $k$ and any crossed product structure.
\end{cor}

Note that Aschenbrenner--Friedl--Wilton asked whether $\Z G$ satisfies the Zero Divisor Conjecture when $G$ is the fundamental group of an irreducible, orientable, manifold with empty or toroidal boundary \cite[Question 7.2.6(6)]{AFW_3mfldBook}, the only unknown case at the time being that of closed graph manifolds. The question was answered positively by Kielak--Linton with their proof of the Strong Atiyah Conjecture for $3$-manifold groups \cite[Corollary 1.2]{KielakLinton_3mfldAtiyah} and \cref{cor:3mfldKap} extends this to positive characteristic. Of course, our result also implies the Zero Divisor Conjecture for torsion-free virtually compact special groups, however this can be deduced from the fact that they have the factorisation property by \cite[Corollary 4.3]{Schreve_AtiyahVCS} and \cite[Theorem 3.7]{FriedlSchreveTillmann_ThurstonFox}. Indeed, since torsion-free virtually compact special groups are residually finite and have the factorisation property, they are fully residually (torsion-free elementary amenable), and group algebras of torsion-free elementary amenable groups satisfy the Zero Divisor Conjecture \cite{KropLinnellMoody_TFEAisOre}. Recall that a group $G$ is said to be \textit{fully residually} $\mathcal P$ (for a property $\mathcal P$ of groups) if for every finite subset $S \subseteq G$ there is a group homomorphism $\varphi \colon G \rightarrow H$ that is injective on $S$ such that $H$ satisfies $\mathcal P$.

It is interesting to remark that in contrast to \cref{cor:3mfldKap}, Gardam's counterexample to Kaplansky's Unit Conjecture was $\mathbb F_2 G$, where $G \cong \Z^3 \rtimes (\Z/2 \oplus \Z/2)$ is isomorphic to the fundamental group of the Hantzsche--Wendt manifold, a flat $3$-manifold \cite{Gardam_units}. The result was extended by Murray who showed that the group algebra $\mathbb F_p G$ also contains non-trivial units for every prime $p$ \cite{murray2021counterexamples}.

\subsection*{Consequences of the existence of \texorpdfstring{$\mathcal D_{kG}$}{}}

In a recent breakthrough, Jaikin-Zapirain and Linton proved that one-relator groups are coherent, \cite{JaikinLinton_oneRelCoherence}, confirming a conjecture of Baumslag \cite{Baumslag_OneRelProblems}, as are their group algebras over fields of characteristic $0$. \cite[Theorem 1.1]{JaikinLinton_oneRelCoherence}. In fact, they showed that $kG$ is coherent whenever $G$ is the fundamental group of a reducible $2$-complex without proper powers \cite[Theorem 1.6]{JaikinLinton_oneRelCoherence} (see the discussion before \cref{cor:coherence} for a definition) and $k$ is a field of characteristic $0$. Note that all torsion-free one-relator groups are fundamental group of such $2$-complexes. The only reason one must assume that $k$ is of characteristic zero is that in this case the Strong Atiyah Conjecture for one-relator groups \cite{JaikinLopez_Atiyah} implies that a Hughes-free division ring $\mathcal D_{kG}$ exists. Thus, combining our construction of $\mathcal D_{kG}$ for $\operatorname{char}(k) > 0$ with Jaikin-Zapirain and Linton's arguments, we obtain the following.

\begin{cor}[\cref{cor:coherence}]
    Let $G$ be a virtually compact special and the fundamental group of a finite reducible $2$-complex. Then any crossed product $k*G$ is coherent for any division ring $k$.
\end{cor}

In another direction, an embedding $kG \hookrightarrow \mathcal D_{kG}$, where $G$ is locally indicable, gives the Hughes-free division ring $\mathcal D_{kG}$ the structure of a $kG$-module and thus we can use it to compute $H_n(G; \mathcal D_{kG})$ and $b_n^{\mathcal D_{kG}}(G) := \dim_{\mathcal D_{kG}} H_n(G; \mathcal D_{kG})$. These are examples of \textit{agrarian invariants} of the group $G$, which were first introduced and studied by Henneke and Kielak in \cite{HennekeKielak_agrarian}. When $k = \Q$ and $G$ is locally indicable, then $\mathcal D_{\Q G}$ always exists and coincides with the division closure of $\Q G$ in $\mathcal U(G)$ mentioned earlier. It then follows that $b_n^{(2)}(G) = b_n^{\mathcal D_{\Q G}}(G)$, so the $L^2$-Betti numbers of $G$ are examples of agrarian invariants.  If $k$ is a field of characteristic $p>0$, then we think of the Betti numbers $b_n^{\mathcal D_{kG}}(G)$ as mod $p$ analogues of the usual $L^2$-Betti numbers. Indeed, in \cite{Fisher_improved} the Betti numbers $b_n^{\mathcal D_{kG}}(G)$ were shown to have many analogous properties to those of $L^2$-Betti numbers, in particular in how they control finiteness properties of kernels of algebraic fibrations and in \cite{FisherHughesLeary_artin} and \cite{AOS_hyperbolization} they were related to the mod $p$ homology growth of $G$.

Recently, Kielak and Linton proved the following embedding theorem for hyperbolic virtually compact special groups.

\begin{thm}[{\cite[Theorem 1.11]{KielakLinton_FbyZ}}]\label{thm:KLmain}
    Let $H$ be hyperbolic and virtually compact special with $\cd_\Q(H) \geqslant 2$. Then, there exists a finite index subgroup $L \leqslant H$ and a map of short exact sequences
    \[
        \begin{tikzcd}
            1 \arrow[r] & K \arrow[d, hook] \arrow[r] & L \arrow[d, hook] \arrow[r] & \Z \arrow[d, Rightarrow, no head] \arrow[r] & 1 \\
            1 \arrow[r] & N \arrow[r]                 & G \arrow[r]                 & \Z \arrow[r]                                & 1
        \end{tikzcd}
    \]
    such that
    \begin{enumerate}[label = (\arabic*)]
        \item $G$ is hyperbolic, compact special, and contains $L$ as a quasi-convex subgroup.
        \item $\cd_\Q(G) = \cd_\Q(H)$.
        \item $N$ is finitely generated.
        \item If $b_p^{(2)}(H) = 0$ for all $2 \leqslant p \leqslant n$, then $N$ is of type $\FP_n(\Q)$.
        \item If $b_p^{(2)}(H) = 0$ for all $p \geqslant 2$, then $\cd_\Q(N) = \cd_\Q(H) - 1$.
    \end{enumerate}
\end{thm}

As a consequence of this result, they are able to show among other things that one-relator groups with torsion are virtually free-by-cyclic \cite[Corollary 1.3]{KielakLinton_FbyZ}, confirming a conjecture of Baumslag. Much of Kielak and Linton's paper is written in the full generality of agrarian homology. However, to prove \cref{thm:KLmain}, they need to restrict themselves to $L^2$-homology as they make crucial use of Schreve's result that virtually compact special groups satisfy the Strong Atiyah Conjecture \cite{Schreve_AtiyahVCS}. Thus, for a torsion-free virtually compact special group $G$, Kielak and Linton can embed the group algebra $\Q G$ into its Linnell division ring $\mathcal D(G)$ and use it in homological arguments to prove \cref{thm:KLmain}. They conjecture \cite[Conjecture 6.7]{KielakLinton_FbyZ} that \cref{thm:KLmain} remains true when every instance of $\Q$ is replaced with an arbitrary field $k$ and every $L^2$-Betti number $b_i^{(2)}$ is replaced with the agrarian Betti number $b_i^{\mathcal D_{kG}}$. Our construction of $\mathcal D_{kG}$ for $G$ torsion-free virtually compact special confirms their conjecture.

\begin{thm}[{\cref{thm:KLmain_agr}}]
    Let $k$ be a division ring. \cref{thm:KLmain} remains true when every instance of $\Q$ is replaced by $k$ and every $L^2$-Betti number $b_i^{(2)}$ is replaced by the agrarian Betti number $b_i^{\mathcal D_{kG}}$.
\end{thm}

\subsection*{Comparison with the Lewin--Lewin division ring}

In \cite{LewinLewinORTF}, Jacques and Tekla Lewin proved that if $k$ is a division ring and $G$ is a torsion-free one-relator group, then $kG$ embeds in a division ring, which we will denote $\overline{kG}$ and call the Lewin--Lewin division ring. This was the first proof that group algebras of torsion-free one-relator groups satisfy the Kaplansky Zero Divisor Conjecture; Brodski\u{\i}'s result that torsion-free one-relator groups are locally indicable \cite{BrodskiiOR}--together with Higman's proof of the Zero Divisor Conjecture for locally indicable groups \cite[Theorem 12]{Higman_units}--gives another proof.

With the proof of the Strong Atiyah Conjecture for locally indicable groups, we know that group algebras of torsion-free one-relator groups have Hughes-free embeddings in characteristic zero and \cref{thm:main} shows that there are Hughes-free embeddings of virtually compact special torsion-free one-relator groups in positive characteristic as well. It is thus natural to compare the constructions of Hughes-free division rings with the Lewin--Lewin division ring. In the final section, we prove the following.

\begin{thm}[\cref{thm:LL_HF}]
    Let $G$ be a torsion-free one-relator group and let $k$ be a division ring such that $kG$ embeds into a Hughes-free division ring $\mathcal D_{kG}$. Then $\overline{kG} \cong \mathcal D_{kG}$ as $kG$-division rings.
\end{thm}

\subsection*{Organisation of the paper}

In \cref{sec:prelims} we recall some notions that will appear throughout the paper and prove some preliminary results about Linnell and Hughes-free division rings. In \cref{sec:GraphsRings} we study graphs of rings and their relationship to the group algebra of a graph of groups; the results proved in this section are geared towards proving that the group algebra of a graph manifold embeds in a division ring. In \cref{sec:divRings} we use Schreve's factorisation property to prove that group algebras of torsion-free virtually compact special groups embed into a division ring. In \cref{sec:3mflds}, we show that the group algebra of a finitely generated torsion-free fundamental group of an orientable $3$-manifold embeds in a division ring; this builds on the results of the two previous sections. In \cref{sec:KLconj} we use our construction of a division ring embedding $kG$ for $G$ torsion-free virtually compact special to confirm \cite[Conjecture 6.7]{KielakLinton_FbyZ}. In \cref{sec:LL}, we prove that the Lewin--Lewin construction of a division ring embedding for the group algebra of a torsion-free one-relator group is Hughes-free whenever a Hughes-free division ring exists. In \cref{sec:appendix}, we adapt the proof of the Zero Divisor Conjecture for orientable $3$-manifold groups to the non-orientable case.

\subsection*{Acknowledgments} 

We are grateful to Andrei Jaikin-Zapirain for many helpful conversations. We would also like to thank Piotr Przytycki and Henry Wilton for answering questions about $3$-manifold topology, and Dawid Kielak, Kevin Schreve, and Eduardo Reyes for valuable comments on our article. We are grateful to an anonymous referee, whose comments and corrections have much improved the article. The first author is supported by the National Science and Engineering Research Council (NSERC) [ref.~no.~567804-2022] and the European Research Council (ERC) under the European Union's Horizon 2020 research and innovation programme (Grant agreement No. 850930). The second author is supported by PID2020-114032GB-I00 of the Ministry of Science and Innovation of Spain.

\section{Preliminaries}\label{sec:prelims}

Throughout, rings are assumed to be associative and unital, and ring homomorphisms preserve the unit.

\subsection{Special groups}

Special cube complexes were introduced by Haglund and Wise in \cite{HaglundWise_special} as a class of nonpositively curved cube complexes that avoid certain pathological hyperplane configurations. One of the core features of a compact special cube complex $X$ is that it admits a $\pi_1$-injective combinatorial local isometry $X \looparrowright S_\Gamma$ to the Salvetti complex $S_\Gamma$ of a finitely generated right-angled Artin group (RAAG) $A_\Gamma$ \cite[Theorem 1.1]{HaglundWise_special}. Thus, fundamental groups of compact special cube complexes inherit many remarkable algebraic properties from RAAGs; for example they are subgroups of $\SL_n(\Z)$, and in particular are residually finite.

Throughout the article, we will refer to groups that are fundamental groups of compact special cube complexes as \textit{compact special groups}. Agol's Theorem \cite[Theorem 1.1]{AgolHaken} states that a hyperbolic group $G$ acting properly and cocompactly on a $\mathrm{CAT}(0)$ cube complex $X$ contains a subgroup $H$ of finite index such that $X/H$ is compact special. In this sense, virtually compact special groups are abundant.

\subsection{Ore domains}

Let $R$ be a ring and let $T \subseteq R$ be a multiplicative set of non-zero divisors. A {\it left ring of fractions} for $R$ with respect to $T$ is a ring $S$ in which $R$ embeds and that satisfies the following:
\begin{enumerate}
    \item Every element of $T$ becomes invertible in $S$;
    \item Every element of $S$ can be written as $t^{-1}r$ for some $t\in T$ and $r\in R$.
\end{enumerate}

In general, there may not exist a left ring of fractions for a given ring $R$. The condition that guarantees such existence is the so-called Ore condition. Let $R$ be a ring and let $T\subseteq R$ be a multiplicative set. We say that $T$ {\it satisfies the left Ore condition} if $Tr\cap Rt \neq \varnothing$ for every $t\in T$ and $r\in R$. If the set $T$ contains only non-zero divisors and satisfies the left Ore condition, then $R$ has a left ring of fractions with respect to $T$ (see, for example, \cite[Theorem 6.2]{GW04}).

\begin{defn}
    Let $R$ be a domain and $T$ the set of all non-zero divisors in $R$. If $T$ satisfies the left Ore condition, then we say that $R$ is a {\it left Ore domain}. In this case, we denote by $\Ore(R)$ its left ring of fractions with respect to $T$ and call it the  {\it Ore localisation} of $R$.
\end{defn}

\subsection{Hughes-free and Linnell division rings} \label{subsec:LinnellHughes}

Let $k$ be a division ring and let $G$ be a group. 

A ring $S$ is $G$-{\it graded} if $S = \bigoplus_{g\in G} S_g$ as an additive group, where $S_g$ is an additive subgroup for every $g\in G$, and $S_g S_h \subseteq S_{gh}$ for all $g,h \in G$. If $S_g$ contains an invertible element $u_g$ for each $g\in G$, then we say that $S$ is a \textit{crossed product} of $S_e$ and $G$ and we shall denote it by $S=S_e * G$. In practice, we will always denote the element $u_g$ by $g$ and view $G$ as a subset of $S*G$ in this way. Note that the usual group ring $RG$ of a group $G$ with $R$ a ring is an example of a crossed product $R*G$.

An $R$-{\it ring} is a pair $(S,\varphi)$ where $\varphi \colon R\rightarrow S$ is a homomorphism. We will often omit $\varphi$ if it is clear from the context. A subring $R \subseteq S$ is called division closed if whenever $r \in R \cap S^\times$ (where $S^\times$ stands for the units of $S$), then $r\inv \in R$. If $R \subseteq S$ is a subring, then we will denote the smallest division closed subring of $S$ containing $R$ by $\Div(R,S)$. An $R$-division ring $\varphi\colon R\rightarrow \mathcal{D}$ is called {\it epic} if $\Div(\varphi(R),\mathcal D) = \mathcal D$.

Let $R \subseteq S$ be an inclusion of rings and let $T \subseteq S$ be a subset. We say that $T$ is \textit{left-linearly independent over $R$} if the sum $R \cdot t_1 + \cdots + R \cdot t_n$ is direct for every finite subset $\{t_1, \dots, t_n\} \subseteq T$.

\begin{defn}\label{def:HF}
    \begin{enumerate}
        \item Let $G$ be a locally indicable group and let $k$ be a division ring. We say that a $k * G$-division ring $\varphi \colon k * G \rightarrow \mathcal D$ is \textit{Hughes-free} if it is injective, epic, and the following linear independence condition is satisfied:
        \begin{enumerate}
            \item[(HF)] whenever $H \leqslant G$ is finitely generated and $N \trianglelefteqslant H$ is such that $H/N = \langle tN \rangle \cong \Z$, then $\{\varphi(t^n) : n \in \Z\}$ is left-linearly independent over $\Div(k*N,\mathcal D)$.
        \end{enumerate}
        In this situation, we say that $\mathcal D$ is a \textit{Hughes-free division ring of fractions} (or simply a \textit{Hughes-free division ring}) of $k * G$.
        
        \item Now suppose $G$ is any torsion-free group. An embedding $\varphi \colon k*G \rightarrow \mathcal D$ is called \textit{Linnell} if it is epic and the following linear independence condition is satisfied:
        \begin{enumerate}
            \item[(L)] if $H \leqslant G$ is a subgroup and $T$ is a right transversal for $H$ in $G$, then $T$ is left-linearly independent over $\Div(k*H, \mathcal D)$.
        \end{enumerate}
        In this situation, we say that $\mathcal D$ is a \textit{Linnell division ring} of $k * G$.
    \end{enumerate}
\end{defn}

\begin{rem}\label{rem:linnell}
    For a torsion-free group $G$ satisfying the Strong Atiyah Conjecture, the division closure $\mathcal D(G)$ of $\C G$ in $\mathcal U(G)$ is a Linnell division ring for $\C G$. In the literature, $\mathcal D(G)$ is called the Linnell division ring of $\C G$, which is the motivation for the terminology of \cref{def:HF}(2).
\end{rem}

Suppose $G$ is locally indicable. Ian Hughes showed that when Hughes-free division rings exist, they are unique up to $k * G$-isomorphism \cite{HughesDivRings1970}. Thus, when it exists, we will always denote the Hughes-free division ring of $k * G$ by $\mathcal D_{k * G}$. If $H$ is a subgroup of $G$, note that the division closure of $k * H$ in $\mathcal D_{k*G}$ is Hughes-free as a $k*H$-division ring, and therefore we have a natural inclusion $\mathcal D_{k*H} \subseteq \mathcal D_{k*G}$. On the other hand, if $G$ is a torsion-free group and there is a Linnell division ring $\mathcal D$ containing $k*G$, then it is not known whether $\mathcal D$ is the only such division ring.

It is clear that if $G$ is locally indicable and $k*G \hookrightarrow \mathcal D$ is a Linnell embedding, then it is also a Hughes-free embedding. We will often use the following surprising recent result of Gr\"ater, which provides a converse. We also refer the reader to \cite[Proposition 2.4]{JaikinLinton_oneRelCoherence} for a proof of the precise statement we are using here.

\begin{thm}[{\cite[Corollary 8.3]{Grater20}}]\label{thm:grater}
    Let $k*G$ be a crossed product of a locally indicable group $G$ with a division ring $k$. If a Hughes-free division ring $\mathcal D_{k*G}$ exists, then it is in fact a Linnell division ring for $k*G$.
\end{thm}

Since RAAGs are residually (torsion-free nilpotent) (\cite{Droms_thesis} and \cite{DuchampKrob_RAAGsRTFN}), so are their subgroups, and in particular so are compact special groups. Thus, Hughes-free division rings exist for compact special groups by the following result of Jaikin-Zapirain.

\begin{thm}[{\cite[Theorem 1.1]{JaikinZapirain2020THEUO}}]\label{thm:HFresults}
    Let $G$ be a locally indicable amenable group, a residually (torsion-free nilpotent) group, or a free-by-cyclic group. Then $\mathcal D_{k * G}$ exists and it is universal.
\end{thm}

We refer the reader to \cref{subsec:slyv} for a definition of universality. We now prove two general lemmas about Hughes-free division rings which will be useful to us in the later sections.

\begin{lem}\label{lem:twisted_ext}
    Let $G$ be a group and let $H \trianglelefteqslant G$ be a locally indicable normal subgroup. Fix a crossed product structure $k*G$ for some division ring $k$. If there is a Hughes-free embedding $\varphi \colon k*H \hookrightarrow \mathcal D_{k*H}$, then we can form $\mathcal D_{k*H} * [G/H]$ and there is a natural embedding $k*G \cong (k*H) * [G/H] \hookrightarrow \mathcal D_{k*H} * [G/H]$.
\end{lem}

\begin{proof}
    The only potential obstruction to extending the crossed product structure of $(k*H) * [G/H]$ to $\mathcal D_{k*H} * [G/H]$ is extending the conjugation action of $G$ on $H$ to a $G$-action on all of $\mathcal D_{k*H}$. This is not a problem, however, because Hughes-free division rings are unique up to $k*H$-isomorphism. In more detail, let $\alpha \colon H \rightarrow H$ be any automorphism of $H$, and by abuse of notation write $\alpha$ for the induced automorphism of $k*H$. Then $\varphi$ and $\varphi \circ \alpha$ are both Hughes-free embeddings of $k * H$, and by uniqueness of Hughes-free embeddings, $\alpha$ extends to an automorphism $\alpha' \colon \mathcal D_{k*H} \rightarrow \mathcal D_{k*H}$ such that the diagram
    \[
        \begin{tikzcd}
            k*H \arrow[d, "\varphi", hook] \arrow[r, "\alpha", hook] & k*H \arrow[d, "\varphi", hook] \\
            \mathcal D_{k*H} \arrow[r, "\alpha'", hook] & \mathcal D_{k*H}              
        \end{tikzcd}
    \]
    commutes. \qedhere
\end{proof}

The following lemma will be key throughout the article, as it allows us to pass the Linnell property to extensions by elementary amenable groups.

\begin{lem}\label{lem:LinnellFIovergroup}
    Let $k*G$ be a crossed product of a division ring $k$ and a torsion-free group $G$. Suppose there is a normal subgroup $H \trianglelefteqslant G$ such that $G/H$ is elementary amenable and a Linnell embedding $k*H \hookrightarrow \mathcal D$. If the conjugation action of $G$ on $H$ extends to an action on $\mathcal D$ and $\mathcal D * [G/H]$ is a domain, then the embedding $k*G \hookrightarrow \Ore(\mathcal D * [G/H])$ is Linnell.
\end{lem}

\begin{proof}
    First note that $\mathcal D *[G/H]$ is an Ore domain by \cite[Lemma 2.5]{LinnellSchick_AtiyahExt}. For the sake of brevity, if $A \leqslant G$, then we write $\mathcal D_A := \Div(k*A, \Ore(\mathcal D *[G/H]))$ and we note that $\mathcal D_A = \Ore(\mathcal D_{H \cap A} * [A/ H \cap A])$. Let $N \leqslant G$ be a subgroup, let $t_1, \dots, t_n$ be distinct right $N$-coset representatives in $G$, and let $\alpha_1, \dots, \alpha_n \in \mathcal D_N$ be such that
    \[
        \alpha_1 t_1 + \cdots + \alpha_n t_n = 0.
    \]
    By multiplying on the left by a common denominator, we may assume that $\alpha_i \in \mathcal D_{H \cap N}*[N/H\cap N]$. Fixing a collection $s_1, \dots, s_k$ of right $H \cap N$-coset representatives in $N$, for each $i$ we can write $\alpha_i = \sum_{l=1}^k \beta_l^i s_l$ for some $\beta_l^i \in \mathcal D_{H \cap N}$. The previous line becomes 
    \[
        \sum_{l=1}^k (\beta_l^1 s_l t_1 + \cdots + \beta_l^n s_l t_n) = 0.
    \]
    Observe that the elements $s_l t_m$ lie in different $H \cap N$-cosets, so the previous line has the form
    \[
        \gamma_1 r_1 + \cdots + \gamma_j r_j = 0
    \]
    where $\gamma_d \in \mathcal D_{H \cap N}$ for each $d$ and $r_1, \dots, r_j$ is a collection of distinct $H \cap N$-coset representatives (here, $j = kn$ and each element $\gamma_d$ is equal to some $\beta_l^i$). Since the $H$-cosets are left-linearly independent over $\mathcal D = \mathcal D_H$ by assumption, it suffices to consider the case where the elements $r_d$ are all contained in the same $H$-coset. But then there is some $g \in G$ such that $r_d = h_d g$ for all $d$, where the elements $h_d \in H$ lie in different $H \cap N$-cosets. We obtain $\sum_d \gamma_d h_d = 0$, implying that $\gamma_d = 0$ for all $d$ by the Linnell property. But then $\alpha_i = 0$ for each $1 \leqslant i \leqslant n$, as desired. \qedhere
\end{proof}

The main situation where we will use the previous lemma is when $H$ is locally indicable and Hughes-free embeddable.

\begin{cor}\label{cor:HF_fi}
    Let $k*G$ be a crossed product of a torsion-free group $G$ and a division ring $k$. Suppose $H \trianglelefteqslant G$ is a normal and locally indicable subgroup such that $G/H$ is elementary amenable. If there is a Hughes-free embedding $k*H \hookrightarrow \mathcal D_{k*H}$ and $\mathcal D_{k*H}*[G/H]$ is a domain, then embedding $k*G \hookrightarrow \Ore(\mathcal D_{k*H}*[G/H])$ is Linnell and is unique among Linnell embeddings up to $k*G$-isomorphism.
\end{cor}

\begin{proof}
    If $\mathcal D_{k*H}$ exists, then it is Linnell by \cref{thm:grater}. By \cref{lem:LinnellFIovergroup}, the embedding $k*G \hookrightarrow \Ore(\mathcal D_{k*H} *[G/H])$ is Linnell. Now suppose that $k*G \hookrightarrow \mathcal D$ is Linnell. Then $\Div(k*H, \mathcal D)$ is Hughes-free, and therefore isomorphic to $\mathcal D_{k*H}$ by \cite{HughesDivRings1970}. Since $\mathcal D$ is Linnell, the cosets of $G/H$ are left-linearly independent over $\mathcal D_{k*H}$ and therefore there is an embedding $\mathcal D_{k*H} * [G/H] \hookrightarrow \mathcal D$, where $G$ acts by conjugation on $\mathcal D_{k*H}$.  By the universal property of Ore localisations, there is a homomorphism $\Ore(\mathcal D_{k*H} *[G/H]) \hookrightarrow \mathcal D$, which is surjective since $\mathcal D$ is epic as a $k*G$-division ring. This proves the uniqueness statement. \qedhere
\end{proof}

\subsection{Agrarian homology}

Let $R$ be a ring, let $G$ be a group, and let $\mathcal D$ be a division ring. If the group algebra $RG$ embeds into $\mathcal D$, then we say that $G$ is \textit{$\mathcal D$-agrarian over $R$} and that the embedding $RG \hookrightarrow \mathcal D$ is an \textit{agrarian embedding}.

Suppose that $G$ is $\mathcal D$-agrarian over $R$. Then $\mathcal D$ is an $RG$-bimodule and we can define and denote the \textit{$\mathcal D$-homology and cohomology} of $G$ by
\[
    H_n (G; \mathcal D) := \Tor_n^{RG}(\mathcal D, R) \quad \text{and} \quad H^n (G; \mathcal D) := \Ext_{RG}^n(R, \mathcal D)
\]
and the \textit{$\mathcal D$-Betti numbers} by
\[
    b_p^\mathcal D(G) := \dim_{\mathcal D} H_p(G; \mathcal D) \quad \text{and} \quad b_\mathcal D^p(G) := \dim_\mathcal D H^p (G; \mathcal D).
\]
The theory of agrarian Betti numbers was introduced by Henneke--Kielak in \cite{HennekeKielak_agrarian} in the case $R = \Z$ and was studied over other fields $R$ in the case where $\mathcal D$ is Hughes-free in \cite{Fisher_improved}. However, in this article we will be mostly concerned with agrarian cohomology. Thanks to \cite[Lemma 2.2]{KielakLinton_FbyZ}, $b_p^\mathcal D(G) = b^p_\mathcal D(G)$ whenever these quantities are finite (which occurs if, for instance, $G$ is of type $\F_\infty$ or more generally of type $\FP_\infty(R)$) and therefore we do not need to worry about the distinction between cohomological and homological $\mathcal D$-Betti numbers.

The following is the central example of agrarian homology.

\begin{ex}
    Let $G$ be a torsion-free group satisfying the Strong Atiyah Conjecture over $k$, where $k$ is a subfield of $\C$. Then $k G$ embeds into a division ring $\mathcal D(G)$ called the \textit{Linnell division ring} (\cite{LinnellDivRings93} and \cite[Lemma 10.39]{Luck02}) and the agrarian Betti numbers $b_p^{\mathcal D(G)}(G)$ are equal to the $L^2$-Betti numbers $b_p^{(2)}(G)$. 
\end{ex}

\begin{prop}\label{prop:props_agr}
    Let $G$ be a group and let $k$ be a division ring such that $kG$ embeds into a division ring $\mathcal D$.
    \begin{enumerate}[label = (\arabic*)]
        \item\label{item:0} If $G$ is non-trivial, then $b_0^\mathcal D(G) = 0$.
        \item\label{item:euler} If $G$ is a group of finite type, then $\chi(G) = \sum_{i = 0}^\infty (-1)^i \cdot b_i^\mathcal D(G)$.
        \item\label{item:scaling} If $G$ is locally indicable and $\mathcal D = \mathcal D_{kG}$ is Hughes-free, then for every finite index subgroup $H \leqslant G$ we have $|G:H| \cdot b_p^{\mathcal D_{kG}}(G) = b_p^{\mathcal D_{kH}} (H)$ for all $p$.
    \end{enumerate}
\end{prop}

\begin{proof}
    \ref{item:0} This follows from considering the partial free resolution 
    \[
        \bigoplus_{g \in G} kG \xrightarrow{\bigoplus_{g \in G} (g - 1)} kG \xrightarrow{\alpha} k \rightarrow 0
    \]
    of the trivial $kG$-module $k$, and tensoring with $\mathcal D$ over $kG$. Here, $\alpha$ denotes the augmentation map.

    \ref{item:euler} This is proved as usual, i.e.~if $C_\bullet(\widetilde X; k)$ is the CW chain complex of $\widetilde X$ with coefficients in $k$, where $X$ is a finite classifying space for $G$, then we use the rank-nullity theorem and the fact that $\dim_\mathcal D \mathcal D \otimes_{kG} C_n(\widetilde X;k)$ is the number of $n$-cells in $X$.

    \ref{item:scaling} This essentially follows from \cref{thm:grater}. For a detailed proof, see \cite[Lemma 6.3]{Fisher_improved}. \qedhere
\end{proof}

\begin{rem}
    \cref{item:scaling} is of \cref{prop:props_agr} is particularly useful, as it allows us to consistently define agrarian Betti numbers for groups $G$ containing a finite index subgroup $H$ such that $kH$ has a Hughes-free embedding.
\end{rem}

\subsection{Sylvester matrix rank functions}\label{subsec:slyv}

Let $R$ be a ring. A \textit{Sylvester matrix rank function} $\rk$ on $R$ is a function that assigns a non-negative real number to each matrix over $R$ and satisfies the following conditions:
\begin{enumerate}
    \item $\rk(A)=0$ if $A$ is any zero matrix and $\rk(1)=1$;
    \item $\rk(AB)\leq \min\{\rk(A),\rk(B)\}$ for any matrices $A$ and $B$ which can be multiplied;
    \item $\rk\begin{psmallmatrix}
         A & 0 \\
         0 & B
    \end{psmallmatrix}=\rk(A)+\rk(B)$ for any matrices $A$ and $B$;
    \item $\rk \begin{psmallmatrix}
         A & C \\
         0 & B
    \end{psmallmatrix} \geq \rk(A)+\rk(B)$ for any matrices $A$, $B$, and $C$ of appropriate sizes.
\end{enumerate}

We denote by $\mathbb{P}(R)$ the set of Sylvester matrix rank functions on $R$. Note that a ring homomorphism $\varphi\colon R\rightarrow S$ induces a map $\varphi^{\#}\colon \mathbb{P}(S)\rightarrow \mathbb{P}(R)$, that is, we can pull back any rank function $\rk$ on $S$ to a rank function $\varphi^{\#}(\rk)$ on $R$ by setting
\[
\varphi^{\#}(\rk)(A):=\rk(\varphi(A))
\]
for every matrix $A$ over $R$. We shall often abuse notation and write $\rk$ instead of $\varphi^{\#}(\rk)$ when it is clear that we are referring to the rank function on $R$.

A division ring $\mathcal{D}$ has a unique Sylvester matrix rank function which we denote by $\rk_{\mathcal{D}}$. Any Sylvester matrix rank function $\rk$ on $R$ that only takes integer values is the pullback of the Sylvester matrix rank function on a division ring by a result of P.~ Malcolmson \cite{Malcolmson}. Furthermore, there is a one-to-one correspondence between integer-valued rank functions and epic $R$-division rings.

\begin{lem}[{\cite[Corollary 3.1.15]{DLopezAlvarezThesis}}]\label{lem:equal_rk}
    Let $R$ be a ring and let $\mathcal{D}$, $\mathcal{E}$ be two epic $R$-division rings. Then $\mathcal{D}$ and $\mathcal{E}$ are $R$-isomorphic if and only if for every matrix $A$ over $R$ the induced rank functions on $R$ satisfy
    \[
        \rk_{\mathcal{D}}(A)=\rk_{\mathcal{E}}(A).
    \]
\end{lem}

We denote the set of integer-valued rank functions on a ring $R$ by $\mathbb{P}_{div}(R)$.

Given two Sylvester matrix rank functions on $R$, $\rk_1$ and $\rk_2$, we will write $\rk_1\leq \rk_2$ if for every matrix $A$ over $R$, $\rk_1(A)\leq \rk_2(A)$. This partial order structure shall play a key role.

A central notion in this theory is that of a universal $R$-division ring for a given ring $R$ (see, for instance, \cite[Section 7.2]{cohn06FIR}). In the language of Sylvester matrix rank functions, an epic $R$-division ring $\mathcal{D}$ is \textit{universal} if $\rk_{\mathcal{D}} \geqslant \rk_{\mathcal{E}}$ for every $R$-division ring $\mathcal{E}$. Note that a universal epic $R$-division ring, if it exists, is unique up to $R$-isomorphism and we denote it by $U(R)$.

\section{Graphs of rings}\label{sec:GraphsRings}

In this section we introduce \textit{graphs of rings} and prove some of their basic properties. The amalgamated product of rings over a common subring has been studied extensively (see, for instance, \cite{cohn06FIR}) and the HNN extension of rings was defined and studied by Dicks in \cite{Dicks_HNN}. The upshot of this section is \cref{cor:locIndGraphDivRing}, which states that crossed products of graphs of Hughes-free embeddable groups embed in a division ring. Our motivation for defining graphs of rings is to prove the Kaplansky Zero Divisor Conjecture for crossed products of fundamental groups of graph manifolds, which are the compact, irreducible $3$-manifolds for which the factorisation property is not known. Moreover, a tree of rings will appear in \cref{sec:LL} when studying the Lewin--Lewin division ring.

We define graphs of rings in complete analogy with graphs of groups. We take graphs to be connected and oriented, with $\overline{e}$ denoting the same edge as $e$ but with the opposite orientation. Every edge $e$ has an origin vertex $o(e)$ and a terminus vertex $t(e)$ such that $o(e) = t(\overline e)$. Graphs are allowed to have loops and multiple edges.

\begin{defn}[The graph of rings with respect to a spanning tree]\label{def:treegraph}
    Let $\Gamma$ be a graph and let $T$ be a spanning tree. For each vertex $v$ of $\Gamma$ we have a \textit{vertex ring} $R_v$ and for each edge $e$ of $\Gamma$ we have an \textit{edge ring} $R_e$ and we impose $R_e = R_{\overline{e}}$ for every edge $e$. Moreover, for each (directed) edge $e$ there is an injective ring homomorphism $\varphi_e \colon R_e \rightarrow R_{t(e)}$. Then the \textit{graph of rings} $\mathscr R_{\Gamma, T} = (R_v, R_e)$ is the ring defined as follows:
    \begin{enumerate}
        \item for each edge of $e$ of $\Gamma$ we introduce a formal symbols $t_e$;
        \item $\mathscr R_{\Gamma, T}$ is generated by the vertex rings $R_v$ and the elements $t_e, t_e\inv$ and subjected to the relations
        \begin{itemize}
            \item $t_{\overline e}t_e = t_e t_{\overline e} = 1$;
            \item $t_e \varphi_{\overline{e}}(r) t_{\overline e} = \varphi_e(r)$ for all $r \in R_e$;
            \item if $e \in T$, then $t_e = 1$. 
        \end{itemize}
    \end{enumerate}
\end{defn}

Define $\mathscr R_\Gamma^*$ in the same way as $\mathscr R_{\Gamma, T}$, except drop the relations $t_e = 1$ if $e \in T$. There is a canonical quotient map $\pi_T \colon \mathscr R_\Gamma^* \rightarrow \mathscr R_{\Gamma, T}$.

\begin{defn}[The based graph of rings]\label{def:basedgraph}
    We retain all the notations of \cref{def:treegraph}. Fix a base vertex $v_0 \in \Gamma$. We say that an element of $\mathscr R_\Gamma^*$ is a \textit{loop element} if it is of the form $r_0 t_{e_1} r_1 t_{e_2} \cdots t_{e_n} r_n$ and
    \begin{enumerate}
        \item $r_0 \in R_{o(e_1)}$;
        \item $r_i \in R_{t(e_i)}$ for all $1 \leqslant i \leqslant n$ ;
        \item $t(e_i) = o(e_{i+1})$ for all $1 \leqslant i \leqslant n-1$;
        \item $o(e_1) = t(e_n) = v_0$.
    \end{enumerate}
    We then define $\mathscr R_{\Gamma, v_0}$ to be the subring of $\mathscr R_\Gamma^*$ generated by the loop elements. Since the product of loop elements is clearly a loop element, $\mathscr R_{\Gamma,v_0}$ consists of the elements of $\mathscr R_\Gamma^*$ that can be expressed as sums of loop elements.
\end{defn}

\begin{rem}
    When defining a ring with generators and relations, we are quotienting a freely generated ring by an ideal. Thus, with these definitions, we of course run the risk that $\mathscr R_\Gamma^*, \mathscr R_{\Gamma, T}$, or $\mathscr R_{\Gamma, v_0}$ is zero, and that we have lost all information about the vertex and edge rings. This never happens in the graph of groups construction, but not much can be said for a general graph of rings. In the situations of interest, however, we will see that this does not happen, and that the vertex rings inject into the graph of rings (see \cref{lem:graphofgrouprings,prop:kGinjects}) as one would hope.
\end{rem}

The following result is the analogue of \cite[Ch.~1, \S 5.2, Proposition 20]{Serre_arbres}. In particular, it implies that the isomorphism types of $\mathscr R_{\Gamma,T}$ and $\mathscr R_{\Gamma,v_0}$ are independent of the choices of $T$ and $v_0$, respectively. We will thus simplify the notation and denote the graph of rings by $\mathscr R_\Gamma$.

\begin{prop}\label{prop:GORiso}
    Restricting the canonical projection $\pi_T \colon \mathscr R_\Gamma^* \rightarrow \mathscr R_{\Gamma, T}$ induces an isomorphism $\alpha := \pi_T|_{\mathscr R_{\Gamma, v_0}} \colon \mathscr R_{\Gamma, v_0} \rightarrow \mathscr R_{\Gamma, T}$.
\end{prop}

\begin{proof}
    The proof is analogous to that of \cite[Ch.~1, \S 5.2, Proposition 20]{Serre_arbres}, to which we refer the reader for more details. For every vertex $v$ of $\Gamma$, let $c_v = e_1 \cdots e_n$ be the geodesic path from $v_0$ to $v$ in $T$ and let $\gamma_v = t_{e_1} \cdots t_{e_n}$ be the corresponding element of $\mathscr R_\Gamma^*$. Put $x' = \gamma_v x \gamma_v\inv$ whenever $x \in R_v$ and $t_e' = \gamma_{o(e)} t_e \gamma_{t(e)}\inv$ for every edge $e$ of $\Gamma$. It is straightforward to show that the assignment $\beta(x) = x'$ and $\beta(t_e) = t_e'$ induces a well-defined homomorphism $\beta \colon \mathscr R_{\Gamma, T} \rightarrow \mathscr R_{\Gamma, v_0}$ such that $\alpha \circ \beta = \id$ and $\beta \circ \alpha = \id$. \qedhere
\end{proof}

When $G$ decomposes as graph of groups $\mathscr G_\Gamma$, the crossed product $k * G$ decomposes as a graph of rings in the expected way.

\begin{lem}\label{lem:graphofgrouprings}
    Let $\mathscr G_\Gamma = (G_v, G_e)$ be a graph of groups with fundamental group $G$ and let $R$ be a ring. Then any crossed product $R * G$ decomposes as a graph of rings $\mathscr R_\Gamma = (R * G_v, R * G_e)$, where the edge maps $R * G_e \rightarrow R * G_{t(e)}$ are induced by the edge maps $G_e \rightarrow G_{t(e)}$ of the graph of groups.
\end{lem}

\begin{proof}
    Let $v_0$ be a vertex in $\Gamma$; we work with the based graph of rings presentation for $\mathscr R_\Gamma$. Define a homomorphism $\alpha \colon R * G \rightarrow \mathscr R_\Gamma$ as follows. Write $g \in G$ as a loop element $g_1 e_1 g_2 e_2 \cdots e_n g_n$ and put $\alpha(g) = g_1 e_1 g_2 e_2 \cdots e_n g_n \in \mathscr R_\Gamma$. This defines a homomorphism of $G$ into the unit group $\mathscr R_\Gamma^\times$, so $\alpha$ extends to a homomorphism $R * G \rightarrow \mathscr R_\Gamma$ by $R$-linearity.

    On the other hand if we put $\beta(g_1 e_1 g_2 e_2 \cdots e_n g_n) = g_1 e_1 g_2 e_2 \cdots e_n g_n \in R * G$ for a loop element $g_1 e_1 g_2 e_2 \cdots e_n g_n$, we also obtain a well-defined homomorphism $\beta \colon \mathscr R_\Gamma \rightarrow R * G$, since the relations in $\mathscr R_\Gamma$ hold in $R * G$ (by the based graph of groups presentation for $G$). \qedhere
\end{proof}

\begin{defn}
    Let $\mathscr G_\Gamma = (G_v, G_e)$ be a graph of torsion-free groups with fundamental group $G$ and fix a division ring $k$ and a crossed product $k * G$. Then $\mathscr G_\Gamma$ is called \textit{$\mathcal D$-compatible} if the following conditions are met:
    \begin{enumerate}[label=(\arabic*)]
        \item For every vertex $v$ of $\Gamma$, there is an embedding $k*G_v \hookrightarrow \mathcal D_v$, where $\mathcal D_v$ denotes a division ring.
        \item\label{item:LI} Let $\mathcal D_e$ denote $\Div(\varphi_e(k*G_e), \mathcal D_{t(e)})$. For all vertices $v$ and all edges $e$ such that $t(e) = v$, any set of right coset representatives of $\varphi_{t(e)} (G_e)$ in $G_v$ is left-linearly independent over $\mathcal D_e$.
        \item\label{item:cong} $\mathcal D_e \cong \mathcal D_{\overline{e}}$ as $k*G_{e}$-division rings for every edge $e$ of $\Gamma$.
    \end{enumerate}
\end{defn}

\begin{rem}
    Condition \ref{item:LI} is automatically satisfied if the embeddings $k*G_v \hookrightarrow \mathcal D_v$ are Linnell. If, in addition, the vertex groups are locally indicable, then condition \ref{item:cong} is automatically satisfied by the uniqueness of Hughes-free division rings \cite{HughesDivRings1970}.
\end{rem}

In what follows, we will usually (by an abuse of notation) denote the fundamental group of a graph of groups $\mathscr G_\Gamma = (G_v, G_e)$ by $\mathscr G_\Gamma$; a choice of base vertex in $\Gamma$ will always be implicit. If $\mathscr G_\Gamma = (G_v, G_e)$ is a $\mathcal D$-compatible graph of groups, then we can form the \textit{graph of division rings} on $\Gamma$ with vertex division rings $\mathcal D_v$ and edge division rings $\mathcal D_e$; we denote it by $\mathscr{DG}_\Gamma$. Our next goal is to prove that $k * \mathscr G_\Gamma$ embeds into $\mathscr{DG}_\Gamma$. For this, we will need the following normal form theorem.

\begin{thm}[{\cite[Theorems 34(i) and 35(i)]{Dicks_HNN}}]\label{thm:NF}
    \begin{enumerate}[label = (\arabic*)]
        \item\label{item:amalgNF} Let $B$ and $C$ be rings containing a common subring $A$ such that $B$ (resp.~$C$) is free as a left $A$-module with basis $\{1\} \sqcup X$ (resp.~$\{1\} \sqcup Y$). Then the amalgam $B *_A C$ is free as a left $B$-module on the set of sequences of strings $y_1 x_1 y_2 x_2 \cdots$ with $x_i \in X$ and $y_i \in Y$ not beginning with an element of $X$ and including the empty sequence.
        \item\label{item:HNNNF} Let $B*_A$ be an HNN extension of rings with stable letter $t$ such that $B$ is free as a left $A$-module under both edge maps, with bases $\{1\} \sqcup X$ and $\{1\} \sqcup Y$.  Then $B*_A$ is free as a left $B$-module on the set of linked expressions constructed from 
        \[
            \begin{matrix}
                \ominus \ X  \ \oplus & \ominus \ Xt^{-1} \cup \{t\inv\} \ \ominus \\
                \oplus \ tY \cup \{t\} \ \oplus & \oplus \ tYt\inv \ \ominus
            \end{matrix}
        \]
        not beginning with an element of $X$ or $Xt\inv$ and including the empty sequence.
    \end{enumerate}
\end{thm}

A \textit{linked expression} is a word $a_1 a_2 a_3 \cdots$ such that if $a_i$ belongs to a set with a $\oplus$ (resp.~$\ominus$) to its right, then $a_{i+1}$ must belong to a set with a $\oplus$ (resp.~$\ominus$) to its left; we refer to \cite{Dicks_HNN} for a precise definition. Note that \ref{item:amalgNF} is deduced from earlier work of Cohn \cite{Cohnfreeproducts} or \cite{Bergman_Coprod}.

In the proof of the following lemma, all transversals that appear are assumed to contain the relevant group's identity element.

\begin{prop}\label{prop:kGinjects}
    Let $k$ be a division ring and let $\mathscr G_\Gamma = (G_v, G_e)$ be a $\mathcal D$-compatible graph of groups (for some fixed crossed product $k * \mathscr G_\Gamma$). Then the natural map $k * \mathscr G_\Gamma \rightarrow \mathscr{DG}_\Gamma$ is an embedding.
\end{prop}

\begin{proof}
    Write $G = \mathscr G_\Gamma$. First assume that $\Gamma$ is finite. We simultaneously prove the following pair of statements by induction on the number of edges in $\Gamma$:
    \begin{enumerate}
        \item $k * \mathscr G_\Gamma \rightarrow \mathscr{DG}_\Gamma$ is an embedding, and
        \item for any vertex $v$ of $\Gamma$, there is a right transversal $T$ of $G_v$ in $G$ such that the image of $T$ in $\mathscr{DG}_\Gamma$ is linearly independent over $\mathcal D_v$.
    \end{enumerate}
    If $\Gamma$ has no edges, then it consists of a single vertex and the claims are trivial. Now suppose that $\Gamma$ has at least one edge. Let $v$ be a vertex of $\Gamma$ and let $e$ be an edge such that $o(e) = v$. Assume that $\Gamma \smallsetminus e$ is disconnected with connected components $\Gamma_1$ and $\Gamma_2$, where $v \in \Gamma_1$. By induction, $k * \mathscr G_{\Gamma_1}$ embeds in $\mathscr{DG}_{\Gamma_1}$ and there is a right transversal $T_1$ of $G_v$ in $\mathscr G_{\Gamma_1}$ which remains linearly independent over $\mathcal D_v$. Let $S_1$ be a right transversal for (the image of) $G_e$ in $G_v$. Then $S_1$ is also linearly independent over $\mathcal D_e$ in $\mathcal D_v$ by $\mathcal D$-compatibility. Thus, $S_1 T_1$ is a right transversal for $G_e$ in $\mathscr G_{\Gamma_1}$ and it is linearly independent over $\mathcal D_e$ in $\mathscr{DG}_{\Gamma_1}$.

    Moreover, we also have that $k * \mathscr G_{\Gamma_2}$ embeds in $\mathscr{DG}_{\Gamma_2}$. By a similar argument, there is a transversal $T_2$ for $G_{t(e)}$ in $\mathscr G_{\Gamma_2}$ and a transversal $S_2$ for $G_e$ in $G_{t(e)}$ such that $T_2$ and $S_2 T_2$ are linearly independent over $\mathcal D_{t(e)}$ and $\mathcal D_e$, respectively.

    Let $X$ be the set of alternating expressions of the form $y_1 x_1 y_2 x_2 \cdots$ with $x_i \in S_1 T_1$ and $y_i \in S_2 T_2$ not beginning with an element of $S_1 T_1$. Note that $X$ is a right transversal for $\mathscr G_{\Gamma_1}$ in $\mathscr G$. By \cref{thm:NF}\ref{item:amalgNF}, we have that $k * G$ embeds in $\mathscr{DG}_\Gamma$ and $X$ is linearly independent over $\mathscr{DG}_{\Gamma_1}$. To complete the induction, note that $T_1 X$ is a right transversal for $G_v$ in $\mathscr G_\Gamma$; so by linear independence of $T_1$ and $X$ over $\mathcal D_v$ and $\mathscr{DG}_{\Gamma_1}$, respectively, we conclude that $T_1 X$ is linearly independent over $\mathcal D_v$.

    The case where $\Gamma \smallsetminus e$ is connected is proved similarly using \cref{thm:NF}\ref{item:HNNNF}; we omit the proof.

    We now drop the assumption that $\Gamma$ is finite. For a contradiction, assume that $k * \mathscr G_\Gamma \rightarrow \mathscr{DG}_\Gamma$ is not injective. Let $x$ be a non-trivial element of the kernel and let $\Gamma' \subseteq \Gamma$ on which $x$ is supported. The image of $x$ in $\mathscr{DG}_\Gamma$ will be a finite linear combination of relators, which are supported in some finite subgraph $\Gamma'' \subseteq \Gamma$. Enlarging $\Gamma'$ and $\Gamma''$ if necessary, we may assume that $\Gamma' = \Gamma''$. But then $x$ is a non-trivial element of the kernel of $k * \mathscr G_{\Gamma'} \rightarrow \mathscr{DG}_{\Gamma'}$, a contradiction. \qedhere
\end{proof}

The main results of this section now follows easily.

\begin{thm}\label{thm:graphOfHF}
    Let $k$ be a division ring and let $\mathscr G_\Gamma = (G_v, G_e)$ be a $\mathcal D$-compatible graph of groups (for some fixed crossed product $k * \mathscr G_\Gamma$). Then $k * \mathscr G_\Gamma$ embeds into a division ring.
\end{thm}

\begin{proof}
    We will prove that $\mathscr{DG}_\Gamma$ is a semifir, namely that all of its finitely generated left (or right) ideals are free of unique rank. The result then follows from \cref{prop:kGinjects} and Cohn's theorem stating that every semifir embeds into a division ring \cite[Corollary 7.5.14]{cohn06FIR}.

    We begin with the case that $\Gamma$ is finite. This follows by induction on the number of edges, using the facts that amalgams of semifirs over a division ring and HNN extensions of a semifir over a division ring are still semifirs (\cite{Cohnfreeproducts} and \cite[Theorems 34(ii) and 35(ii)]{Dicks_HNN}. 
    
    If $\Gamma$ is infinite, then the result follows since $\mathscr{DG}_\Gamma$ is the colimit of the semifirs $\mathscr{DG}_{\Gamma'}$ with $\Gamma'$ finite (see \cite[\S1.1 Exercise 3]{Cohn_FreeRingsRelations}). \qedhere
\end{proof}

\begin{cor}\label{cor:locIndGraphDivRing}
    Let $k$ be a division ring, let $\mathscr G_\Gamma = (G_v, G_e)$ be a graph of locally indicable groups, and fix a crossed product $k * \mathscr G_\Gamma$. Suppose there is a Hughes-free embedding $k * G_v \hookrightarrow \mathcal D_{k * G_v}$ for each vertex $v$ of $\Gamma$. Then $k * \mathscr G_\Gamma$ embeds in a division ring.
\end{cor}
\begin{proof}
    By \cref{thm:grater}, Hughes-free embeddings are in fact Linnell embeddings \cite[Corollary 8.3]{Grater20}. Recall that if $A \leqslant B$ are groups and there is a Hughes-free embedding $k * B \hookrightarrow \mathcal D_{k*B}$, then $\Div(k*A, \mathcal D_{k*B})$ is isomorphic to the unique Hughes-free division ring $\mathcal D_{k*A}$ (this follows from the uniqueness of Hughes-free division rings \cite{HughesDivRings1970}). Thus, $\mathscr G_\Gamma$ is $\mathcal D$-compatible. \qedhere
\end{proof}

\begin{cor}\label{cor:atiyahGraph}
    Let $k$ be a subfield of $\C$ and let $\mathscr G_\Gamma = (G_v, G_e)$ be a graph of (torsion-free groups satisfying the Strong Atiyah Conjecture over $k$). Then the group algebra $k\mathscr G_\Gamma$ embeds in a division ring.
\end{cor}

\begin{proof}
    Since the vertex groups satisfy the Strong Atiyah Conjecture, $k G_v$ embeds into the Linnell division ring $\mathcal D(G_v)$. The fact that $\mathscr G_\Gamma$ is $\mathcal D$-compatible follows from the fact that if $B$ is a torsion-free group satisfying the Strong Atiyah Conjecture and $A \leqslant B$, then the division closure of $k A$ in $\mathcal D(B)$ is isomorphic to $\mathcal D(A)$ (see either \cite[Proposition 4.6]{KielakBNSviaNewton} or \cite[Chapter 10]{Luck02}).
\end{proof}

\begin{q}\label{q:HF}
    Is the embedding constructed in \cref{cor:locIndGraphDivRing} Hughes-free when the graph of groups $\mathscr G_\Gamma$ is locally indicable?
\end{q}

In \cref{lem:HF_amalg}, we will see that a positive answer to \cref{q:HF} is equivalent to the existence of a Hughes-free division ring.

We conclude the section with a short application of our main result. Recall that Higman's group $H$ can be defined by the presentation
\[
    \langle a,b,c,d \mid b^a = b^2, c^b = c^2, d^c = d^2, a^d = a^2 \rangle.
\]
It can also be realised as a square of groups with $\mathrm{BS}(1,2)$ vertex groups, $\Z$ edge groups, and trivial face group. The group $H$ was constructed by Higman in \cite{Higman_higmangroup}, and it was the first example of an infinite group with no non-trivial finite quotients. It also has an infinite simple quotient.

Rivas and Triestino showed that Higman's group acts faithfully and continuously on $\R$, and therefore is left-orderable and in particular $RH$ satisfies Kaplansky's Zero Divisor Conjecture for all domains $R$ \cite[Theorem A, Corollary B]{RivasTriestino_Higman}. Here we show that $kH$ has the (a priori) stronger property of embedding into a division ring, at least when $k$ is a field of characteristic zero.

\begin{prop}
    Let $k$ be a field of characteristic $0$. Then $kH$ embeds into a division ring.
\end{prop}

\begin{proof}
    Indeed, note that $H$ decomposes as $H = G_1 *_A G_2$, where
    \[
        G_1 = \langle a,b,c \rangle, \quad A = \langle a,c \rangle \cong F_2, \quad G_2 = \langle a,c,d \rangle.
    \]
    Moreover, for $i = 1,2$, we have that $G_i \cong \BS(1,2) *_\Z \BS(1,2)$ is a cyclic amalgam of locally indicable groups, and therefore is locally indicable by a result of Howie \cite[Theorem 4.2]{Howie_locIndGps}. Hence, $kG_i$ has a Hughes-free embedding for $i=1,2$ by \cite[Corollary 1.4]{JaikinLopez_Atiyah} and therefore $kH$ embeds into a division ring by \cref{cor:locIndGraphDivRing}.
\end{proof}

\section{Division rings for groups with the factorisation property} \label{sec:divRings}

The following definition was introduced by Schreve in \cite{Schreve_AtiyahVCS}, where he used it to show that virtually compact special groups satisfy the Strong Atiyah Conjecture. This is a strengthening of the \textit{enough torsion-free quotients} property introduced by Linnell and Schick in \cite{LinnellSchick_AtiyahExt}, which they used to study when the Strong Atiyah Conjecture passes from a subgroup to a finite index overgroup.

\begin{defn}
    A group $G$ has the \textit{factorisation property} if for every homomorphism $\alpha \colon G \rightarrow Q$ with $Q$ finite, there is a torsion-free elementary amenable group $E$ such that $\alpha$ factors as $G \rightarrow E \rightarrow Q$.
\end{defn}

Recall that a residually finite group $G$ is \textit{good} (in the sense of Serre) if the restriction 
\[
    H_c^\bullet(\widehat G; M) \rightarrow H^\bullet (G;M)
\]
is an isomorphism for every finite $G$-module $M$, where $\widehat G$ denotes the profinite completion of $G$, and the cohomology on the left is continuous cohomology (see, for example, \cite{Serre_GaloisCohom}).

The proof of the following theorem is very close to the arguments of Linnell and Schick in \cite[Corollary 4.62]{LinnellSchick_AtiyahExt}. The main difference consists in replacing Linnell and Schick's cohomological completeness and enough torsion-free quotients conditions with Schreve's goodness and factorisation property conditions, respectively.

\begin{thm}\label{thm:VCS_field}
    Let $H$ be a locally indicable good group of finite type with the factorisation property and let $1 \rightarrow H \rightarrow G \rightarrow Q \rightarrow 1$ be a group extension with $G$ torsion-free and $Q$ finite. If $k$ is a division ring such that there is a Hughes-free embedding of $k * H$ into $\mathcal D_{k * H}$, then $k * G$ has a unique Linnell embedding into a division ring.
\end{thm}

\begin{proof}
    By \cite[Theorem 3.7]{FriedlSchreveTillmann_ThurstonFox}, $G$ has the factorisation property and therefore there is a normal subgroup $U \trianglelefteqslant G$ such that $U \leqslant H$ and $G/U$ is torsion-free and elementary amenable. By \cref{lem:twisted_ext}, we can form each of the following rings:
    \[
        \mathcal D_{k * U} * [H/U], \quad \mathcal D_{k * U} * [G/U], \quad \mathcal D_{k * H} * [G/H].
    \]
    Since $H/U$ and $G/U$ are torsion-free elementary amenable, according to \cite[Lemma 2.5]{LinnellSchick_AtiyahExt} $\mathcal D_{k * U} * [H/U]$ and $\mathcal D_{k * U} * [G/U]$ are Ore domains and so the diagram
    \[
        \begin{tikzcd}
            {\mathcal D_{k * U} * [H/U]} \arrow[r, hook] \arrow[d, hook] & {\mathcal D_{k * U} * [G/U]} \arrow[d, hook] \\
            {\Ore(\mathcal D_{k * U} * [H/U])} \arrow[r, hook] &{\Ore(\mathcal D_{k * U} * [G/U])}          
        \end{tikzcd}
    \]
    commutes. By Hughes-freeness of $\mathcal D_{k * H}$, the map $\mathcal D_{k * U} * [H/U] \rightarrow \mathcal D_{k * H}$ is an injection. This implies that $\Ore(\mathcal D_{k * U} * [H/U]) \cong \mathcal D_{k * H}$ by the universal property of Ore localisation.

    Consider the following diagram:
    \[
    \begin{tikzcd}[column sep = small]
        \mathcal D_{k * H} \arrow[r, "\cong", no head] \arrow[d, hook] & {\Ore(\mathcal D_{k * U} * [H/U])} \arrow[r, hook] \arrow[d, hook] & {\Ore(\mathcal D_{k * U} * [G/U])} \arrow[d, "\cong", no head] \\
        {\mathcal D_{k * H} *[G/H]} \arrow[r, "\cong", no head]        & {\Ore(\mathcal D_{k * U} * [H/U]) * [G/H]} \arrow[r, "\cong", no head]        & {\Ore((\mathcal D_{k * U}*[H/U])*[G/H])} \nospacepunct{.}                    
    \end{tikzcd}
    \]
    The left and middle vertical maps are the obvious inclusions and the right vertical map is a standard isomorphism of crossed products. The two left isomorphisms come from the isomorphism $\Ore(\mathcal D_{k * U} * [H/U]) \cong \mathcal D_{k * H}$ discussed above. For the bottom right isomorphism, it is not hard to show that the natural map
    \[
        \Ore(\mathcal D_{k * U} * [H/U]) * [G/H] \rightarrow \Ore((\mathcal D_{k * U}*[H/U])*[G/H]),
    \]
    is injective. Therefore $\Ore(\mathcal D_{k * U} * [H/U]) * [G/H]$ is a domain, which implies it is a division ring since $G/H$ is finite. This proves that $\mathcal D_{k * H} * [G/H]$ is a division ring, which clearly contains $k * G$. Moreover, the embedding is Linnell and unique among Linnell embeddings by \cref{cor:HF_fi}. \qedhere
\end{proof}

\begin{cor}\label{cor:CS_by_EA_field}
    Let $H$ be a locally indicable good group of finite type with the factorisation property and let $1 \rightarrow H \rightarrow G \rightarrow A \rightarrow 1$ be a group extension with $G$ torsion-free and $A$ elementary amenable. If $k$ is a division ring such that there is a Hughes-free embedding of $k * H$ into $\mathcal D_{k * H}$, then $k * G$ embeds into a division ring.
\end{cor}

\begin{proof}
    By Hughes-freeness, the twisted action of $A$ on $k * H$ extends to a twisted action of $A$ on $\mathcal{D}_{k * H}$. Moreover, $\mathcal{D}_{k * H} * Q$ is a domain for every finite subgroup $Q$ of $A$ by the proof of \cref{thm:VCS_field}. Thus, according to \cite[Lemma 2.5]{LinnellSchick_AtiyahExt} $\mathcal{D}_{k * H} * A$ is an Ore domain. Therefore, $\mathcal{D}_{k * H} * A$, and hence also $(k * H) * A \cong K * G$, embeds into a division ring.
\end{proof}

\begin{cor}\label{cor:VCSinField}
    If $G$ is a torsion-free virtually compact special group and $k$ a division ring, then any crossed product $k * G$ has a unique Linnell embedding into a division ring $\mathcal D$. Moreover, if $H$ is a normal, finite index, compact special subgroup of $G$, then the diagram
    \[
        \begin{tikzcd}
            k * H \arrow[r, hook] \arrow[d, hook] & {(k * H) * [G/H]} \arrow[d, hook] \arrow[r, "\cong", no head] & k * G \arrow[d, hook] \\
            \mathcal D_{k * H} \arrow[r, hook]    & {\mathcal D_{k * H} * [G/H]} \arrow[r, "\cong", no head]    & \mathcal D        
        \end{tikzcd}
    \]
    commutes.
\end{cor}

\begin{proof}
    Compact special groups are residually (torsion-free nilpotent) and therefore $\mathcal D_{k*H}$ exists by \cref{thm:HFresults}. Moreover, compact special groups are good and have the factorisation property \cite[Corollary 4.3]{Schreve_AtiyahVCS} and are of finite type, since they have finite classifying spaces. Hence there is an embedding of $k*G$ into a division ring $\mathcal D \cong \mathcal D_{k*H} * [G/H]$ by \cref{thm:VCS_field}. The embedding $k*G \hookrightarrow \mathcal D$ is Linnell and unique by \cref{cor:HF_fi}. \qedhere
\end{proof}

In \cite[Theorem 1.6]{JaikinLinton_oneRelCoherence}, Jaikin-Zapirain and Linton prove that $kG$ is coherent whenever $k$ is of characteristic $0$ and $G$ is the fundamental group of a finite reducible $2$-complex without proper powers. The class $\mathcal R$ of all fundamental groups of finite reducible $2$-complexes without proper powers can be described algebraically as the smallest class of groups containing $\Z$, that is closed under free product, and satisfying the following property: if $G, H \in \mathcal R$, then $G*H/\llangle w \rrangle \in \mathcal R$, where $w$ is an element of $G*H$ that is not a proper power.

Crucially, we recall Howie's result that the fundamental group of reducible $2$-complexes without proper powers are locally indicable \cite[Theorem 4.2]{Howie_locIndGps}. The only reason the assumption that $k$ is of characteristic $0$ is needed is to ensure that a Hughes-free division ring $\mathcal D_{kG}$ exists, which follows from the fact that one-relator groups satisfy the Strong Atiyah Conjecture \cite{JaikinLopez_Atiyah}. From Jaikin-Zapirain and Linton's arguments and the existence of $\mathcal D_{k*G}$ we obtain the coherence of some $2$-complex fundamental group algebras in positive characteristic.

\begin{cor}\label{cor:coherence}
    Let $G \in \mathcal R$ and let $k$ be a division ring. If $G$ is virtually compact special, then any crossed product $k*G$ is coherent.
\end{cor}

We remark that though the results of Jaikin-Zapirain and Linton are stated for group algebras, all of their techniques work just as well in the crossed product case.

\section{Division rings for \texorpdfstring{$3$}{3}-manifold groups}\label{sec:3mflds}

Manifolds are always assumed to be connected. The goal of this section is to prove the following embedding theorem.

\begin{thm}\label{thm:3mfld}
    Let $G$ be the fundamental group of an orientable $3$-manifold $M$, let $k$ be a division ring, and assume that $G$ is torsion-free. If $G$ is finitely generated, then any crossed product $k*G$ embeds into a division ring. If $G$ is locally indicable, any crossed product $k*G$ has a Hughes-free embedding.
\end{thm}

\begin{rem}
    \cref{thm:3mfld} remains true if we drop the assumption of orientability, though we will first prove the theorem as stated here in order to keep the exposition as straightforward as possible. In \cref{sec:appendix}, we will show how the methods used here can be extended to the non-orientable case by using the non-orientable versions of the Prime and JSJ Decomposition Theorems.
\end{rem}

\begin{cor}\label{cor:3mfldKapl}
    Let $M$ be an orientable $3$-manifold with torsion-free fundamental group $G$ and let $k$ be a division ring. Then $k * G$ has no zero divisors.
\end{cor}

\begin{proof}
    By \cref{thm:3mfld}, it follows that $k*H$ has no zero divisors for any finitely generated subgroup of $G$. But then $k*G$ has no zero divisors, since any two elements of $k*G$ have support contained in a finitely generated subgroup of $G$. \qedhere
\end{proof}

\begin{rem}
    The Strong Atiyah Conjecture over $\C$ is now known for all finitely generated $3$-manifold groups and all torsion-free $3$-manifold groups, as they all lie in Linnell's class $\mathcal C$ (see \cite[Theorem 1.5]{LinnellDivRings93}). This follows from the work of Friedl--L\"uck and Kielak--Linton \cite{FriedlLuck_euler,KielakLinton_3mfldAtiyah}, building on the results on $3$-manifold fibring of Agol, Liu, Przytycki, and Wise \cite{AgolHaken,Liu_npcGraph,PrzWise_GraphSpecial,PrzWise_MixedSpecial}. It follows that \cref{thm:3mfld,cor:3mfldKapl} are known for the group algebra $\C \pi_1(M)$ where $M$ is a torsion-free $3$-manifold with finitely generated fundamental group.
\end{rem}

We offer the following extension of \cite[Theorem 3.3]{FriedlSchreveTillmann_ThurstonFox}, where it is shown that the $3$-manifold groups below have the factorisation property provided the boundary is empty or toroidal. We also show that in all cases where the factorisation property is satisfied, then the crossed product embeds into a unique Linnell division ring.

\begin{thm}\label{thm:nonGraphDivRings}
    Let $k$ be a division ring and let $M$ be a compact, irreducible, aspherical $3$-manifold with (possibly empty) incompressible boundary. Suppose that
    \begin{enumerate}[label=(\arabic*)]
        \item $M$ has non-empty boundary,
        \item $M$ has at least one hyperbolic piece in its JSJ decomposition, or
        \item $M$ is a non-positively curved graph manifold or is Nil, Sol, or Seifert fibred.
    \end{enumerate}
    Then $G = \pi_1(M)$ has the factorisation property and $k * G$ has a unique Linnell embedding into a division ring.
\end{thm}

\begin{proof}
    First assume that $M$ has non-empty incompressible boundary. Since $M$ is aspherical, $M$ is a $K(G,1)$ space and therefore $\cd_\Q(G) < 3$, implying that there is a finite index subgroup $H \trianglelefteqslant G$ that is isomorphic to a free-by-cyclic group (i.e.~an extension $F \rtimes \Z$, where $F$ is a free group) by \cite[Theorem 3.3]{KielakLinton_3mfldAtiyah}. By \cite[Lemma 2.4]{Schreve_AtiyahVCS} free-by-cyclic groups have the factorisation property, and \cite[Proposition 3.7]{FriedlSchreveTillmann_ThurstonFox} implies that $G$ also has the factorisation property. The subgroup $H$ is Hughes-free embeddable by \cite[Theorem 1.1]{JaikinZapirain2020THEUO}. By \cref{cor:HF_fi,thm:VCS_field}, $k * G$ has a unique Linnell embedding into a division ring.

    If $M$ has a hyperbolic piece in its JSJ decomposition, then $M$ is virtually fibred by the work of Agol and Przytycki--Wise \cite{AgolHaken,PrzWise_MixedSpecial}. In particular, $G$ has a subgroup of finite index isomorphic to $\pi_1(\Sigma) \rtimes \Z$, where $\Sigma$ is a surface (potentially with boundary). Hence, $G$ has the factorisation property by \cite[Propositions 3.6 and 3.7]{FriedlSchreveTillmann_ThurstonFox}. If $\partial \Sigma \neq \varnothing$, then $\pi_1(\Sigma)$ is free and, as we showed in the previous paragraph, $k * G$ has a unique Linnell embedding. If $\Sigma$ is closed, then $\pi_1(\Sigma)$ is free-by-cyclic, so it is Hughes-free embeddable. Therefore $\pi_1(\Sigma) \rtimes \Z$ is Hughes-free embeddable by \cite{HughesDivRings1970_2}. Thus, we conclude that $k * G$ can be uniquely embedded in a Linnell division ring by \cref{thm:VCS_field} and \cref{lem:LinnellFIovergroup}.

    If $M$ is a non-positively curved graph manifold, then it is virtually fibred over the circle by a theorem of Svetlov \cite{Svetlov_npcGraph} (which was strengthened by Liu \cite{Liu_npcGraph}). The desired conclusions then follow from the same arguments as above. If $M$ is Nil, Sol, or Seifert fibred, then $G$ has the factorisation property by \cite[Lemmas 3.8 and 3.9]{FriedlSchreveTillmann_ThurstonFox}. The fundamental groups of Nil and Sol manifolds are torsion-free elementary amenable, and therefore $k * G$ is an Ore domain according to \cite[Lemma 2.5]{LinnellSchick_AtiyahExt}. So $\Ore(k * G)$ is a Linnell division ring for $k*G$.

    Finally, we treat the case where $M$ is a closed Seifert fibred manifold. The argument is inspired by \cite[Theorem 3.2(3), argument (c)]{FriedlLuck_euler}. By \cite[Flowchart 1, (C.19)]{AFW_3mfldBook}, there is a locally indicable subgroup $H \trianglelefteqslant G$ of finite index. Let $N \rightarrow M$ be the corresponding finite cover. Let $\varphi \colon H \rightarrow \Z$ be a map, and let $\widetilde N \rightarrow N$ be the infinite cyclic cover. Using the Scott Core Theorem \cite{Scott_core}, we write $\widetilde N = \bigcup_i N_i$ as a directed union of $\pi_1$-injective Seifert fibred manifolds $N_i$ with boundary. As we have seen above, each $k * \pi_1(N_i)$ embeds in a Hughes-free division ring $\mathcal D_{k * \pi_1(N_i)}$. It then follows that $k* \pi_1(\widetilde N) = \bigcup_i k * \pi_1(N_i)$ embeds in the Hughes-free division ring 
    \[
        \mathcal D_{k * \pi_1(\widetilde N)} = \bigcup_i \mathcal D_{k * \pi_1(N_i)}.
    \]
    Then $H \cong \pi_1(\widetilde N) \rtimes \Z$ is Hughes-free embeddable by \cite{HughesDivRings1970_2}. Since $G$ has the factorisation property \cite[Theorem 3.3]{FriedlSchreveTillmann_ThurstonFox}, $k * G$ embeds in a unique Linnell division ring by \cref{cor:HF_fi,thm:VCS_field}. \qedhere
\end{proof}

Finally, we turn to the case of a general graph manifold $M$ with torsion-free fundamental group $G$. If $G$ surjects onto $\Z$, then we can use the same argument as in the last paragraph of the proof of \cref{thm:nonGraphDivRings} to produce an embedding into a division ring. While it is true in general that $G$ has a virtual map onto $\Z$, we cannot use this in conjunction with \cref{thm:VCS_field} to produce an embedding into a division ring, since $G$ is not known to have the factorisation property. Thus, we take a different approach than the one used in \cref{thm:nonGraphDivRings}, which uses the graph of rings construction introduced in \cref{sec:GraphsRings}.

\begin{thm}\label{thm:graphMfld}
    Let $k$ be a division ring and let $M$ be a closed graph manifold with torsion-free fundamental group $G$. Then any crossed product $k * G$ embeds into a division ring. If $G$ is locally indicable, then $k*G$ has a Hughes-free embedding.
\end{thm}

\begin{proof}
    By definition, the JSJ decomposition of $M$ has only Seifert fibred pieces. If $M$ has only one JSJ component, then the claim follows from \cref{thm:nonGraphDivRings}.

    Now suppose that $M$ has JSJ components $M_1, \dots, M_n$ with $n \geqslant 2$. The manifolds $M_i$ are all Seifert fibred manifolds with nonempty toroidal boundary. By \cite[Flowchart 1, (C.19)]{AFW_3mfldBook}, $\pi_1(M_i)$ is locally indicable for each $i$, and therefore $k * \pi_1(M_i)$ embeds into a Hughes-free division ring $\mathcal D_{k * \pi_1(M_i)}$ by \cref{thm:nonGraphDivRings}. Since $G$ decomposes as a graph of groups with Hughes-free embeddable vertex groups, \cref{cor:locIndGraphDivRing} implies that $k * G$ embeds into a division ring.

    Now suppose that $G$ is locally indicable, let $\varphi \colon G \rightarrow \Z$ be any epimorphism, and let $\widetilde M \rightarrow M$ be the corresponding infinite cyclic covering. Then $\widetilde M$ is a directed union of $\pi_1$-injective graph manifolds with boundary whose fundamental groups are all Hughes-free embeddable by \cref{thm:nonGraphDivRings}. Hence, $\pi_1(\widetilde M)$ is Hughes-free embeddable, and therefore so is $G \cong \pi_1(\widetilde M) \rtimes \Z$ by \cite{HughesDivRings1970_2}. \qedhere
\end{proof}

\begin{rem}
    Let $G$ be as above and let $H \trianglelefteqslant G$ be a locally indicable normal subgroup of finite index (see \cite[Flowchart 1, (C.19)]{AFW_3mfldBook} for a justification of the existence of such a subgroup). Then we know that $k*G$ embeds into a division ring $\mathcal D$ constructed using the graph of rings construction, and we also know that $k*H$ has a Hughes-free embedding into $\mathcal D_{k*H}$. However, we do not know how to prove that $\mathcal D \cong \mathcal D_{k*H} * [G/H]$, which is what prevents us from concluding that the embedding $k*G \hookrightarrow \mathcal D$ is Linnell.
\end{rem}

The final ingredient we will need to prove \cref{thm:3mfld} is the following form of the Prime Decomposition Theorem \cite[Theorem 1.2.1 and Lemma 1.4.2]{AFW_3mfldBook}. We refer the reader to \cite[Proposition 3.2]{KielakLinton_3mfldAtiyah} for a detailed proof. Note that since we are assuming orientability and torsion-freeness, we can actually take the whole group in the conclusion, as opposed to a finite index subgroup (as in \cite{KielakLinton_3mfldAtiyah}).

\begin{prop}\label{prop:3decomp}
    Let $M$ be an orientable $3$-manifold whose fundamental group $G$ is finitely generated and torsion-free. Then there is a free group $F$ and finitely many compact, orientable, irreducible, aspherical $3$-manifolds $M_1, \dots, M_n$ each with (possibly empty) incompressible boundary such that
    \[
        G \cong F * \pi_1(M_1) * \cdots * \pi_1(M_n).
    \]
\end{prop}

We are now ready to prove the main theorem.

\begin{proof}[Proof of \cref{thm:3mfld}]
    Suppose that $G$ is finitely generated and let $G \cong F * \pi_1(M_1) * \cdots * \pi_1(M_n)$ be the decomposition given by \cref{prop:3decomp}. The crossed product $k * F$ embeds into $\mathcal D_{k * F}$, and the crossed products of the fundamental groups $\pi_1(M_i)$ all embed into division rings $\mathcal D_i$ by \cref{thm:nonGraphDivRings,thm:graphMfld}. Thus, $k * G$ embeds into the amalgam of $\mathcal D_{k * F}$ and the division rings $\mathcal D_i$ over the common sub-division ring $k$, which in turn embeds into a division ring by \cref{thm:graphOfHF} (in fact \cite{Cohnfreeproducts_3} suffices in this case). 

    Now suppose that $G$ is locally indicable. Then each crossed product $k*\pi_1(M_i)$ embeds into a Hughes-free division ring, and therefore so does $k*G$ by \cite[Corollary 6.13(iv)]{JSanchezThesis}. In the case where $G$ is locally indicable but not necessarily finitely generated, then every finitely generated subgroup of $G$ is a finitely generated locally indicable $3$-manifold and is thus Hughes-free embeddable. But then $G$ is the directed union of Hughes-free embeddable groups and is therefore itself Hughes-free embeddable. \qedhere
\end{proof}

\section{A conjecture of Kielak and Linton} \label{sec:KLconj}

Let $G$ be a torsion-free virtually compact special group. Since $G$ is not necessarily locally indicable, a Hughes-free division ring $\mathcal D_{kG}$ may not exist. However, by \cref{prop:props_agr}, we may simply define $b_p^{\mathcal D_{kG}}(G) := \frac{1}{|G:H|} b_p^{\mathcal D_{kH}}(H)$, where $H \leqslant G$ is a compact special subgroup of finite index. We record the following easy observation that will be useful later.

\begin{lem}\label{lem:eitherField}
    Let $k$ be a division ring, let $G$ be a torsion-free virtually compact special group and let $\mathcal D$ be the division ring containing $kG$ constructed in \cref{thm:VCS_field}. Then $b_p^\mathcal D(G) = b_p^{\mathcal D_{kG}}(G)$ for all $p$.
\end{lem}

\begin{proof}
    Let $H \trianglelefteqslant G$ be a compact special subgroup of finite index. The claim then follows quickly from the definition above and the fact that $\mathcal D \cong \mathcal D_{kH} * [G/H] \cong \bigoplus_{|G:H|}\mathcal D_{kH}$ as left $\mathcal D_{kH}$-modules. \qedhere
\end{proof}

\begin{lem}\label{lem:pair_of_free}
    Let $G$ be a non-elementary hyperbolic group, let $k$ be a division ring, and suppose there is an embedding of $kG$ into a division ring $\mathcal D$. Then there exists non-isomorphic quasi-convex free subgroups $H,F \leqslant G$ such that $F$ is malnormal, $H \cap F^g = \{1\}$ for all $g \in G$, and the restriction
    \[
        H^1(G; \mathcal D) \rightarrow H^1(H; \mathcal D)
    \]
    is an isomorphism.
\end{lem}

\begin{proof}
    This is \cite[Corollary 5.7]{KielakLinton_FbyZ}, and the proof is similar: it only relies on an ``agrarian Freiheitssatz" which they prove for arbitrary agrarian embeddings \cite[Theorem 3.1]{KielakLinton_FbyZ} based on the $L^2$-Freiheitssatz of Peterson--Thom \cite[Corollary 4.7]{PetersonThom_Freiheit}. \qedhere
\end{proof}

\begin{lem}\label{lem:agr_ind}
    Let $G = A*_C$ where $A$ and $C$ are locally indicable and finitely generated. Moreover, suppose that $k$ is such that $\mathcal D_{kA}$ exists and $kG$ embeds into a division ring $\mathcal D \supseteq \mathcal D_{kA}$ making the diagram
    \[
        \begin{tikzcd}
            kA \arrow[r, hook] \arrow[d, hook] & kG \arrow[d, hook] \\
            \mathcal D_{kA} \arrow[r, hook]    & \mathcal D
        \end{tikzcd}
    \]
    commute. If the restriction $H^1(A; \mathcal D_{kA}) \rightarrow H^1(C; \mathcal D_{kA})$ is surjective, then the restriction
    \[
        H^2(G; \mathcal D) \rightarrow H^2(A; \mathcal D)
    \]
    is injective.
\end{lem}

\begin{proof}
    Let $H \leqslant G$ and write $\mathcal D(H)$ for the division closure of $kH$ in $\mathcal D$. Then the proof is the same as in \cite[Proposition 4.8]{KielakLinton_FbyZ}, where one must replace every occurrence of $\mathcal D_{\Q H}$ with $\mathcal D(H)$. \qedhere
\end{proof}

The following proposition corresponds to Proposition 6.4 of \cite{KielakLinton_FbyZ}, where the result is proven in the case $k = \Q$. The obstacle Kielak and Linton faced in their paper was the fact that they didn't have access to division rings containing group rings of torsion-free virtually compact special groups. We repeat their proof, since this is the crucial step where the existence of a division ring embedding the group algebra of a torsion-free virtually compact special group is used.

\begin{prop}\label{prop:HNN}
    Let $H$ be non-free, torsion-free, hyperbolic, and compact special, and suppose that $b_1^{\mathcal D_{kH}}(H) \neq 0$. Then there is a hyperbolic and virtually compact special HNN extension $G = H*_F$ such that the embeddings of $F$ are quasi-convex in $G$ and such that
    \[
        b_p^{\mathcal D_{kG}}(G) = \begin{cases}
            0 & \text{if} \ p = 1 \\
            b_p^{\mathcal D_{kH}}(H) & \text{if} \ p \neq 1.
        \end{cases}
    \]
    Moreover, $H$ is quasi-convex in $G$ and $\cd_k(G) = \cd_k(H)$.
\end{prop}

\begin{proof}
    By \cref{lem:pair_of_free}, there is a pair of isomorphic free quasi-convex subgroups $A,B \leqslant H$ such that $A$ is malnormal and intersects every conjugate of $B$ trivially and the restriction 
    \[
        H^1(H; \mathcal D_{kH}) \rightarrow H^1(B; \mathcal D_{kH})
    \]
    is an isomorphism.
    
    Now let $f \colon A \rightarrow B$ be any isomorphism and let $G = H*_A$ be the corresponding HNN extension. By \cite[Theorem 6.3]{KielakLinton_FbyZ}, $G$ is virtually compact special. Moreover, since $G$ is the HNN extension of a torsion-free group, it is also torsion-free, and therefore $kG$ embeds in a division ring $\mathcal D$ by \cref{cor:VCSinField}. From \cite[Theorem 3.1]{BieriHNN}, there is a long exact sequence
    \[
        \cdots \rightarrow H^p(G; \mathcal D) \rightarrow H^p(H; \mathcal D) \rightarrow H^p(A; \mathcal D) \rightarrow H^{p+1}(G; \mathcal D) \rightarrow \cdots.
    \]
    Since $A$ is free, $H^p(A; \mathcal D) = 0$ for $p \geqslant 2$, and this immediately implies that $b_p^{\mathcal D_{kG}}(G) = b_p^{\mathcal D_{kH}}(H)$ for $p \geqslant 3$, where we have used \cref{lem:eitherField}.

    The interesting portion of the long exact sequence is then
    \[
        0 \rightarrow H^1(G; \mathcal D) \rightarrow H^1(H; \mathcal D) \rightarrow H^1(A; \mathcal D) \rightarrow H^2(G; \mathcal D) \rightarrow H^2(H; \mathcal D) \rightarrow 0,
    \]
    whence we obtain the equation
    \begin{align*}
        0 &= b_1^{\mathcal D_{kG}}(G) - b_1^{\mathcal D_{kH}}(H) + b_1^{\mathcal D_{kA}}(A) - b_2^{\mathcal D_{kG}}(G) + b_2^{\mathcal D_{kH}}(H) \\
        &= b_1^{\mathcal D_{kG}}(G) - b_2^{\mathcal D_{kG}}(G) + b_2^{\mathcal D_{kH}}(H).
    \end{align*}
    But $b_2^{\mathcal D_{kG}}(G) \leqslant b_2^{\mathcal D_{kH}}(H)$ by \cref{lem:agr_ind}, so we must have $b_1^{\mathcal D_{kG}}(G) = 0$ and $b_2^{\mathcal D_{kG}}(G) = b_2^{\mathcal D_{kH}}(H)$.

    The claim about cohomological dimensions follows exactly as in the proof of \cite[Proposition 6.4]{KielakLinton_FbyZ}.
\end{proof}

As a consequence, we can reprove \cite[Theorem 1.10]{KielakLinton_FbyZ} over arbitrary fields, thus confirming \cite[Conjecture 6.7]{KielakLinton_FbyZ}.

\begin{thm}\label{thm:KLmain_agr}
    Let $k$ be a division ring and let $H$ be hyperbolic, virtually compact special, and suppose that $\cd_k(H) \geqslant 2$. Then, there exists a finite index subgroup $L \leqslant H$ and a map of short exact sequences
    \[
        \begin{tikzcd}
            1 \arrow[r] & K \arrow[d, hook] \arrow[r] & L \arrow[d, hook] \arrow[r] & \Z \arrow[d, Rightarrow, no head] \arrow[r] & 1 \\
            1 \arrow[r] & N \arrow[r]                 & G \arrow[r]                 & \Z \arrow[r]                                & 1
        \end{tikzcd}
    \]
    such that
    \begin{enumerate}[label = (\arabic*)]
        \item $G$ is hyperbolic, compact special, and contains $L$ as a quasi-convex subgroup.
        \item $\cd_k(G) = \cd_k(H)$.
        \item $N$ is finitely generated.
        \item If $b_p^{\mathcal D_{kH}}(H) = 0$ for all $2 \leqslant p \leqslant n$, then $N$ is of type $\FP_n(k)$.
        \item If $b_p^{\mathcal D_{kH}}(H) = 0$ for all $p \geqslant 2$, then $\cd_k(N) = \cd_k(H) - 1$.
    \end{enumerate}
\end{thm}

\begin{proof}
    Now that we have established \cref{prop:HNN} (Kielak and Linton prove the $L^2$ case in \cite[Proposition 6.4]{KielakLinton_FbyZ}), the proof is very similar to the one in \cite{KielakLinton_FbyZ}; consequently, we only highlight which parts of $L^2$-theory are used in the proof and provide references for the corresponding statements in the agrarian setting.

    First, Kielak--Linton use the fact that an infinite group has vanishing $0^{th}$ $L^2$-Betti number; this is also true in the agrarian setting by \cref{prop:props_agr}\ref{item:0}. Next, they use the fact that $L^2$-Betti numbers scale with the index when passing to finite index subgroups, which is true for agrarian Betti numbers with Hughes-free coefficients by \cref{prop:props_agr}\ref{item:scaling}. Finally, they quote \cite[Theorem A]{Fisher_improved} (see \cite[Theorem 6.1]{KielakLinton_FbyZ}), which relates the $L^2$-Betti numbers of compact special groups (and more generally of RFRS groups) to virtual fibring with kernels of type $\FP_n(\Q)$. Luckily, the analogous result holds with agrarian Betti numbers and finiteness properties over arbitrary fields \cite[Theorem B]{Fisher_improved}. \qedhere
\end{proof}

\section{Hughes-freeness of the Lewin--Lewin division ring} \label{sec:LL}

In \cite{LewinLewinORTF} J.~Lewin and T.~Lewin showed that for every division ring $k$ and every torsion-free one-relator group $G$, the group algebra $kG$ can be embedded in a division ring, which we denote by $\overline{kG}$, following their notation. They pointed out that they did not know whether $\overline{kG}$ is a universal $kG$-division ring of fractions. However, what they already knew is that if there were a universal $kG$-division ring of fractions, then it would be $\overline{kG}$. In this section we give evidence in this direction by showing that the Lewin--Lewin division ring is Hughes-free for virtually compact special one-relator groups, and if $k$ is of characteristic zero then it is always Hughes-free.

Key algebraic structures in the argument of Lewin--Lewin are firs and semifirs. A nonzero ring $R$ is a {\it free ideal ring} (or {\it fir}) if every left and every right ideal is a free $R$-module of unique rank. {\it Semifirs} are defined similarly, the only difference being that only ask finitely generated left and right ideals to be free of unique rank. We will make heavy use of the following two powerful theorems of Cohn.

\begin{thm}[\cite{Cohnfreeproducts_3}]\label{thm:2semifir_semifir}
    The coproduct of a family of semifirs over a common division subring is again a semifir.
\end{thm}

We recall that an epic $R$-division ring $\mathcal{D}$ is universal if $\rk_{\mathcal{D}} \geqslant \rk_{\mathcal{E}}$ for every $R$-division ring $\mathcal{E}$. Whenever it exists, we will denote it by $U(R)$.

\begin{thm}[{\cite[Corollary 7.5.14]{cohn06FIR}}]\label{thm:semifirDivRing}
    Every semifir embeds into a universal division ring of fractions.
\end{thm}

We also record two useful results for future reference.

\begin{lem}[\cite{Cohnfreeprodskewfields}]\label{lem:univ_equality}
    Let $R_1,R_2$ be semifirs with a common division subring $\mathcal{D}$. Then $U(R_1 *_{\mathcal{D}} R_2)\cong U(R_1 *_{\mathcal{D}} U(R_2))$.
\end{lem}

\begin{prop}[{\cite[Theorem 8.1]{JaikinLopez_Atiyah}}]\label{prop:rk_max_HF}
    Let $G$ be a locally indicable group and assume there exists a Hughes-free $kG$-division ring $\mathcal{D}_{kG}$. Then the Sylvester matrix rank function $\rk_{\mathcal{D}_{kG}}$ is maximal in $\mathbb{P}_{div}(kG)$.
\end{prop}

\begin{rem}
    Note that \cref{prop:rk_max_HF} does not imply that Hughes-free division rings are necessarily universal, since different rank functions may not be comparable. However, if there exists both a universal and a Hughes-free division rings of fractions for $k*G$, ring then they must coincide
\end{rem}

We now prove two general lemmas about Hughes-free division rings which will be used right afterwards.

\begin{lem}\label{lem:HF_amalg}
    Let $G$ be a locally indicable group which splits as a graph of groups $\mathscr G_\Gamma = (G_v, G_e; \Gamma)$. If $\mathcal{D}_{k*G}$ exists, then
    \[
        U(\mathscr{DG}_\Gamma) \cong \mathcal{D}_{k*G}
    \]
    as $kG$-rings, where $\mathscr{DG}_\Gamma$ is the graph of division rings $(\mathcal D_{kG_v}, \mathcal D_{kG_e}; \Gamma)$.
\end{lem}

\begin{proof}
    Recall, from \cref{sec:GraphsRings}, that $\mathscr{DG}_\Gamma$ is well-defined by uniqueness of Hughes-free division rings and it is a semifir and thus has a universal division ring. The embeddings of the division rings $\mathcal{D}_{k*G_v}$ and $\mathcal{D}_{k*G_e}$ into $\mathcal{D}_{k*G}$ induce a homomorphism $\mathscr{DG}_\Gamma \rightarrow \mathcal D_{k*G}$, which fits into the following commutative diagram
    \[
        \begin{tikzcd}
            kG \arrow[r, hook] \arrow[d, hook] & {\mathscr{DG}_\Gamma} \arrow[d, hook] \arrow[dl]\\
            \mathcal{D}_{k*G}   & U(\mathscr{DG}_\Gamma) \nospacepunct{.}  
        \end{tikzcd}
    \]
    Let $\rk_{U(G)}$ and $\rk_G$ denote the Sylvester matrix rank functions on $\mathscr{DG}_\Gamma$ corresponding to $U(\mathscr{DG}_\Gamma)$ and $\mathcal{D}_{k*G}$, respectively. Then $\rk_{U(G)}\geq \rk_{G}$ by universality. 
    
    On the other hand, $\mathcal{D}_{kG}$ is Hughes-free and hence $\rk_{G}$ is a maximal rank function on $kG$. Thus, $\rk_{U(G)} = \rk_{G}$ on $kG$ by \cref{prop:rk_max_HF}. Since $kG$ embeds in both division rings epically, they are $kG$-isomorphic by \cref{lem:equal_rk}. \qedhere
\end{proof}

\begin{lem}\label{lem:HF_cyclic}
    Let $G$ be a locally indicable group and $N\trianglelefteqslant G$ be a normal subgroup such that $G/N=\langle tN \rangle \cong \Z$. If $\mathcal{D}_{k*G}$ is a Hughes-free $k*G$-division ring of fractions, then 
    $$\mathcal{D}_{k*G}\cong \Ore(\mathcal{D}_{k*N}[t^{\pm 1};\tau])$$
    as $k*G$-rings.
\end{lem}

\begin{proof}
    First note that since $\mathcal{D}_{k*G}$ is a Hughes-free division ring, the integer powers of $t$ are left $\mathcal{D}_{k*N}$-linearly independent. Hence, if we set $S$ for the subring of $\mathcal{D}_{k*G}$ generated by $\mathcal{D}_{k*N}$, $t$, and $t^{-1}$, we get an isomorphism $S\cong \mathcal{D}_N[t^{\pm 1};\tau]$ where $\tau$ denotes the twisted action of $t$ on $\mathcal D_{k*N}$ induced by conjugation. Note that $S$ is an Ore domain. Thus, by universal property of localisation, we have the following commutative diagram
    \[
        \begin{tikzcd}
             \Ore(\mathcal{D}_N[t^{\pm 1};\tau]) \arrow[rr, hook] &&  \mathcal{D}_{k*G}\\
            &kG \arrow[ul, hook'] \arrow[ur, hook]   &    
        \end{tikzcd}
    \]
    Finally, since both embeddings are epic, we conclude that $\mathcal{D}_{k*G}$ and $\Ore(\mathcal{D}_N[t^{\pm 1};\tau])$ are $k*G$-isomorphic.\qedhere
\end{proof}

We now state the main result of this section.

\begin{thm}\label{thm:LL_HF}
    Let $G$ be a torsion-free one-relator group and let $k$ be a division ring such that $kG$ embeds into a Hughes-free division ring $\mathcal D_{kG}$. Then $\overline{kG} \cong \mathcal D_{kG}$ as $kG$-division rings.
\end{thm}

\begin{const}[Lewin--Lewin]\label{const:LL}
Before proving \cref{thm:LL_HF}, we recall the Lewin--Lewin construction of the division ring $\overline{kG}$ (see \cite{LewinLewinORTF} for the full details). Let $G=\langle a_1, a_2,\ldots \mid R \rangle$ be a torsion-free one-relator group, where $R$ is a cyclically reduced word in the generators $a_i$. The construction of $\overline{kG}$ is done by induction on the complexity of $R$, where the \textit{complexity} of $R$ is defined to be the length of $R$ minus the number of generators appearing in $R$. Specifically, they proved that there is an embedding of $kG$ into a division ring $\overline{kG}$ such that any Magnus subgroup satisfies the Linnell condition and the division closure of its group algebra is universal.

If $R$ has complexity $0$, then $G$ is free, and $\overline{kG}$ is just the universal division ring of fractions $U(kG)$ (cf.~\cite{Lewinfree}).

Now assume the complexity is greater than zero. We also assume that every generator appears in $R$. Otherwise $G$ decomposes as a free product $H * F$, where $H = \langle a_1, \dots, a_n \mid R \rangle$ with $R$ involving all the generators $a_1, \dots, a_n$ and $F$ is a free group; then $\overline{kG}$ is defined to be $U(\overline{kH} *_k U(kF))$. 

Assume for now that $R$ involves a generator of exponent sum zero; we explain how Lewin--Lewin reduce to this case at the end. Without loss of generality, suppose that $t = a_1$ has exponent sum zero, and let $N$ be the normal subgroup generated by the elements $a_i$ for $i \geqslant 2$. Note that $N$ splits as a line of groups
\[
    \cdots * N_{i-1} \underset{A_{i-1,i}}{*} N_i \underset{A_{i,i+1}}{*} N_{i+1} * \cdots 
\]
where each $N_i$ is a one-relator group with relator of complexity strictly less than that of $R$ and each $A_{i-1,i}$ is a Magnus subgroup (see \cite[Section 5]{LewinLewinORTF}). By induction, each ring $kN_i$ embeds into a division ring $\overline{kN_i}$ satisfying the Linnell condition on Magnus subgroups and containing the universal division ring of fractions of $kA_{i-1,i}$ and $kA_{i,i+1}$, respectively. By uniqueness of the universal division ring of fractions, Lewin--Lewin conclude that the line of groups is $\mathcal D$-compatible (to use the terminology of \cref{sec:GraphsRings}). Hence, we can form the following line of division rings
\[
    C = \cdots * \overline{kN_{i-1}} \underset{\overline{kA_{i-1,i}}}{*} \overline{kN_i} \underset{\overline{kA_{i,i+1}}}{*} \overline{kN_{i+1}} * \cdots .
\]
which is a semifir by Cohn's theorem, and thanks to the Linnell condition it contains the group ring
\[
    kN \cong \cdots * kN_{i-1} \underset{kA_{i-1,i}}{*} kN_i \underset{kA_{i,i+1}}{*} kN_{i+1} * \cdots 
\]
Thus $\overline{kN}$ is defined as the universal division ring of fractions $U(C)$. Finally, conjugation by $t$ induces an automorphism on $\overline{kN}$ according to \cite[Theorem 2]{LewinLewinORTF}, and $\overline{kG}$ is the Ore localisation of $\overline{kN}[t^{\pm 1};\tau]$. This concludes the construction whenever $R$ involves a generator with exponent sum zero.

Suppose now that $R$ involves no generator with exponent sum zero. If this is the case, then by \cite[Proposition 2]{LewinLewinORTF} the one relator group $H := G * \Z$ can be given a one-relator presentation where the cyclically reduced word has complexity strictly less than that of $R$. Lewin--Lewin set $\overline{kG}$ to be the division closure of $kG$ inside $\overline{kH}$, which exists by induction.
\end{const}

We point out that the core of Lewin--Lewin's work consists in showing that $\overline{kG}$ enjoys the desired properties on Magnus subgroups, namely that the Magnus subgroups of $kN_i$ satisfy the Linnell condition in $\overline{kN_i}$ and $\overline{kN_i}$ contains the universal division ring of fractions of the Magnus subgroups of $N_i$ for all $i$. This is what allows them to pass the argument up through the induction. Moreover, it is worthwhile to record that this proof was prior to the discovery of the fact that torsion-free one-relator groups are locally indicable due to Brodski\u{\i} in \cite{BrodskiiOR}.

\begin{proof}[{Proof of \cref{thm:LL_HF}}]
    The plan of the proof is to follow \cref{const:LL} and use our assumption that $\mathcal D_{kG}$ exists to prove that $\overline{kG} \cong \mathcal D_{kG}$. Crucially, we recall Brodski\u{\i}'s result that torsion-free one-relator groups are locally indicable \cite{BrodskiiOR}. Let $G=\langle a_1, a_2,\ldots \mid R \rangle$ be a torsion-free one-relator group, where $R$ is a cyclically reduced word in the generators $a_i$. 
    
    We argue, as in \cite{LewinLewinORTF}, by induction on the complexity of $R$. If $R$ has complexity $0$, then $G$ is free and $U(kG) \cong \mathcal D_{kG}$ (see, e.g., \cite[Theorem 1.1]{JaikinZapirain2020THEUO}).
    
    Now assume the complexity is greater than zero. We also assume that every generator appears in $R$. Otherwise $G$ decomposes as a free product $H * F$, where $H = \langle a_1, \dots, a_n \mid R \rangle$ with $R$ involving all the generators $a_1, \dots, a_n$ and $F$ is a free group. In this case, the Lewin--Lewin construction is $\overline{kG} = U(\overline{kH} *_k U(kF))$. So assuming the result for $H$, we can conclude that $\overline{kG} \cong \mathcal D_{kG}$ by \cref{lem:HF_amalg}. 
    
    We also assume that $R$ involves a generator of exponent sum zero. Otherwise, $\overline{kG}$ is the division closure of $kG$ inside $\overline{kH}$  (where $H = G * \Z$ as above). By induction on the complexity, $\overline{kH} \cong \mathcal D_{kH}$, and hence $\overline{kG} \cong \mathcal D_{kG}$. Thus, without loss of generality, suppose that $t = a_1$ has exponent sum zero, and let $N$ be the normal subgroup generated by the elements $a_i$ for $i \geqslant 2$. Then $N$ splits as a line of groups
    \[
        \cdots * N_{i-1} \underset{A_{i-1,i}}{*} N_i \underset{A_{i,i+1}}{*} N_{i+1} * \cdots 
    \]
    where each $N_i$ is a one-relator group with relator of complexity strictly less than that of $R$ and each $A_{i-1,i}$ is a free group. By induction, each $\overline{kN_i} \cong \mathcal D_{kN_i}$, so
    \[
        C \cong \cdots * \mathcal D_{kN_{i-1}} \underset{\mathcal D_{kA_{i-1,i}}}{*} \mathcal D_{kN_i} \underset{\mathcal D_{kA_{i,i+1}}}{*} \mathcal D_{kN_{i+1}} * \cdots.
    \]
    Note that, by uniqueness of Hughes-free division rings, we can also construct a natural embedding of $kN$ into the division ring $\mathcal D$ given by the direct limit of the system 
    \[
        \{ U(\mathcal{D}_{kN_{-i}} *_{\mathcal{D}_{kA_{-i,-i+1}}} \cdots *_{\mathcal{D}_{kA_{i-1,i}}} \mathcal{D}_{kN_{i}})\}_{i\in \N}
    \]
    of Hughes-free division rings by \cref{lem:HF_amalg}. Note $\mathcal D$ is Hughes-free as a $kN$-division ring and therefore $\mathcal D = \mathcal D_{kN}$ coincides with the division closure of $kN$ in $\mathcal D_{kG}$. Moreover, $U(C) \cong \mathcal D_{kN}$ by \cref{lem:HF_amalg} and thus $\overline{kN} \cong \mathcal{D}_{kN}$.
    
    Finally, $\overline{kG}$ is the Ore localisation of $\overline{kN}[t^{\pm 1};\tau]$, which according to Lemma \ref{lem:HF_cyclic} is $kG$-isomorphic to the Hughes-free $kG$-division ring of fractions $\mathcal{D}_{kG}$. This concludes the proof of the case where $R$ involves a generator with exponent sum zero, and hence the theorem. \qedhere 
\end{proof}

\begin{rem}
    The above proof shows that the Lewin--Lewin construction can be further extended to crossed products $k * G$.
\end{rem}

\begin{cor}\label{cor:LLHF}
    Let $G$ be a torsion-free one-relator group and let $k$ be a field of characteristic $0$. Then $\overline{kG}$ is Hughes-free. If $k$ is a general division ring, then $\overline{kG}$ is Hughes-free when $G$ is virtually compact special.
\end{cor}

\begin{proof}
    Recall that torsion-free one-relator groups are locally indicable \cite{BrodskiiOR}. Now, for locally indicable groups over a characteristic $0$ field there always exists $\mathcal{D}_{kG}$ according \cite[Corollary 1.4]{JaikinLopez_Atiyah}. Therefore, $\overline{kG}$ is Hughes-free by \cref{thm:LL_HF}.

    Similarly, from \cref{cor:VCSinField} follows the existence of $\mathcal{D}_{kG}$ in the virtually compact special case. Another use of \cref{thm:LL_HF} ends the proof.
\end{proof}

In a sense made precise below, most one-relator groups are virtually compact special. The following definition is due to Puder \cite{Puder_PR}.

\begin{defn}\label{def:PR}
    The \textit{primitivity rank} of a word $w$ in a free group $F$ is the minimal rank of a subgroup $K$ containing $w$ such that $w$ is imprimitive in $K$ (we define the primitivity rank to be $\infty$ if there is no such $K$).
\end{defn}

Note that words of primitivity rank $1$ are exactly the proper powers. Let $G$ be a one-relator group with presentation $G = F/\llangle R \rrangle$. Louder and Wilton proved that $G$ has negative immersions if and only if $R$ has primitivity rank at least $3$ \cite[Theorem 1.3]{LouderWilton_NegativeImmersions} and Linton showed that one-relator groups with negative immersions are virtually compact special \cite[Theorem 8.2]{Linton_ORH}. Thus, $\mathcal D_{kG}$ exists by \cref{cor:VCSinField} and \cref{lem:LinnellFIovergroup}, and by \cref{thm:LL_HF} $\overline{kG} \cong \mathcal{D}_{kG}$ whenever $R$ is of primitivity rank at least $3$. An advantage of the characterisation in terms of primitivity rank is that, given a one-relator group $G$, there is a relatively simple algorithm to check whether $G$ has negative immersions and therefore is virtually compact special \cite[Corollary 4.4]{Puder_PR}.

We conclude with the following natural question. By the previous remarks, it is settled in the affirmative in characteristic $0$, and only the primitivity rank $2$ case remains in characteristic $p > 0$.

\begin{q}
    Is the Lewin--Lewin construction always Hughes-free?
\end{q}

\appendix \section{The Kaplansky Zero Divisor Conjecture for non-orientable \texorpdfstring{$3$}{}-manifold groups} \label{sec:appendix}

To prove Kaplansky's Zero Divisor Conjecture for group algebras of torsion-free orientable $3$-manifold groups, we made heavy use of the Prime and JSJ Decomposition Theorems, which are usually stated for orientable $3$-manifolds. These theorems have been extended to non-orientable $3$-manifolds by Epstein \cite{Epstein_nonoriprime} and Bonahon--Siebenmann \cite{BonahonSiebenmann_nonoriJSJ}, respectively, and we explain here how to apply them to prove that the group algebra of any torsion-free $3$-manifold group satisfies Kaplansky's Zero Divisor Conjecture.

We will be using the statements of the non-orientable decomposition theorems as found in Bonahon's survey \cite[Theorems 3.1, 3.2, and 3.4]{Bonahon_3mfldsSurvey}, so we take some time to ensure that our definitions are in line with his. The non-orientable Prime Decomposition Theorem tells us how to cut a $3$-manifold along essential spheres and projective planes, while the non-orientable JSJ Decomposition tells us how to cut a $3$-manifold containing no essential spheres or projective planes along essential $2$-tori and Klein bottles. An \textit{essential sphere} in a $3$-manifold $M$ is an embedded copy of $S^2$ such that neither component of $M \smallsetminus S^2$ is homeomorphic to a $3$-ball. An \textit{essential projective plane} in a $3$-manifold $M$ is an embedded $2$-sided copy of $\R P^2$. We say that a $3$-manifold $M$ is \textit{irreducible} if does not contain any essential spheres or projective planes. Since $G = \pi_1(M)$ is always assumed to be torsion-free, we will not have to worry about the projective planes as the next lemma shows. We include a short proof for the reader's convenience.

\begin{lem}\label{lem:noProj}
    Let $M$ be a $3$-manifold and let $\R P^2 \hookrightarrow M$ be an embedding. Then the induced homomorphism $\pi_1(\R P^2) \rightarrow \pi_1(M)$ is injective.
\end{lem}

\begin{proof}
    The embedded $\R P^2$ lifts to a disjoint collection $\Sigma$ of $2$-spheres and projective planes in the universal cover $\widetilde M$. If $\Sigma$ contains a copy of $S^2$, then we are done, because then a loop in $M$ representing the generator of $\R P^2$ has a lift to a non-closed path in $\widetilde M$. If $\Sigma$ contains a copy of $\R P^2$, then the Scott Core Theorem implies that there is a compact, simply connected $3$-manifold containing a copy of $\R P^2$. Simply-connected manifolds are orientable, and therefore so are their boundaries. Thus we can fill the boundary to conclude that there is an embedding of $\R P^2$ into a simply-connected, closed $3$-manifold. By the Poincar\'e Conjecture, we have thus produced an embedding $\R P^2 \hookrightarrow S^3$, a contradiction. \qedhere
\end{proof}

An \textit{essential torus} in a $3$-manifold $M$ is an embedded $\pi_1$-injective copy of $T^2$. Bonahon defines an \textit{essential Klein bottle} to be an embedded copy of the Klein bottle $K$ such that the composition $T^2 \rightarrow K \hookrightarrow M$ is $\pi_1$-injective, where $T^2 \rightarrow K$ is the orientation double cover. We will use the following as the definition of an essential Klein bottle.

\begin{lem}
    An embedded Klein bottle $\iota \colon K \hookrightarrow M$ is essential if and only if $\iota$ is $\pi_1$-injective.
\end{lem}

\begin{proof}
    If $\iota$ is $\pi_1$-injective, then clearly $K$ is essential in $M$. Conversely, since $T^2 \rightarrow K \hookrightarrow M$ is $\pi_1$-injective, either $\iota_* \pi_1(K)$ contains $\Z^2$ as an index $2$ subgroup in which case $\iota$ is $\pi_1$-injective, or $\iota_* \pi_1(K) \cong \Z^2$. But $H_1(K) \cong \Z \oplus \Z/2$, which rules out the second case. \qedhere
\end{proof}

The following lemma is well-known; we include a proof because we have not seen it appear without the assumption of orientability. In the orientable case, the lemma essentially follows from \cite[Theorem 6.1]{Howie_locIndGps}.

\begin{lem}\label{lem:locInd3mfld}
    Let $M$ be a compact, irreducible $3$-manifold whose boundary $\partial M$ contains a component of Euler characteristic $\leqslant 0$. Then $\pi_1(M)$ is locally indicable.
\end{lem}

\begin{proof}
    In the proof, all homology groups will take coefficients in $\Q$. We first prove that $M$ has $b_1(M) > 0$. If $M$ is orientable, then the Half Lives-Half Dies Lemma states that the kernel of the inclusion induced map $H_1(\partial M) \rightarrow  H_1(M)$ has dimension equal to $\frac{1}{2} \dim H_1(\partial M)$. In particular, $b_1(M) > 0$. If $M$ is not orientable, then let $p \colon \widetilde M \rightarrow M$ be the orientation double cover, and let $\tau \colon H_1(M) \hookrightarrow H_1(\widetilde M)$ be the (injective) transfer homomorphism. The following diagram commutes
    \[
        \begin{tikzcd}
            H_1(\partial M) \arrow[d, "\tau", hook] \arrow[r] & H_1(M) \arrow[d, "\tau", hook] \\
            H_1(\partial \widetilde{M}) \arrow[r]            & H_1(\widetilde M) \nospacepunct{,}
        \end{tikzcd}
    \]
    and the bottom map is nonzero by the Half Lives-Half Dies Lemma. Moreover, the the image of $\tau \colon H_1(\partial M) \hookrightarrow H_1(\partial \widetilde M)$ is not contained in the kernel of $H_1(\partial \widetilde M) \rightarrow H_1(\widetilde M)$, which is nontrivial on each non-spherical boundary component. We conclude that $H_1(\partial M) \rightarrow H_1(M)$ is nonzero, and therefore $b_1(M) > 0$. We have thus shown that $\pi_1(M)$ has a homomorphism to $\Z$.

    Let $G \leqslant \pi_1(M)$ be a finitely generated group. If $G$ is of finite-index in $\pi_1(M)$, then the injectivity of the transfer homomorphism implies that $b_1(G) > 0$. Now suppose that $G$ is of infinite index and let $N \rightarrow M$ be the corresponding non-compact covering space. Let $L$ be a core of $N$ obtained by the Scott Core Theorem. Note that $L$ necessarily contains some non-($S^2$ or $\R P^2$) boundary component, and therefore by the work above $b_1(G) = b_1(N) = b_1(L) > 0$. \qedhere
\end{proof}

\begin{thm}\label{thm:nonOriIrred}
    Let $M$ be a compact and irreducible $3$-manifold with torsion-free fundamental group $G$. Then any crossed product $k*G$ embeds in a division ring, which is Hughes-free if $G$ is locally indicable.
\end{thm}

\begin{rem}
    If the orientation double cover $\overline{M} \rightarrow M$ is not a closed graph manifold, then $\pi_1(\overline M)$ is good and has the factorisation property, and therefore the arguments of \cref{sec:3mflds} can be used to conclude \cref{thm:nonOriIrred}. However, in the case where $\overline M$ is a closed graph manifold, a non-orientable version of the proof of \cref{thm:graphMfld} is needed.
\end{rem}

\begin{proof}
    By the non-orientable version of the JSJ Decomposition Theorem due to Bonahon--Siebenmann \cite[Splitting Theorem 1]{BonahonSiebenmann_nonoriJSJ} (see \cite[Theorem 3.4]{Bonahon_3mfldsSurvey} for the form in which we are using the theorem), $M$ can be cut along two-sided essential tori and Klein bottles such that every component either is Seifert fibred or contains no essential $2$-tori or Klein bottles. Let $M_1, \dots, M_n$ be the JSJ components and $G_1, \dots, G_n$ their fundamental groups. Each $G_i$ has a finite index subgroup having the factorisation property by \cref{thm:nonGraphDivRings} and therefore has the factorisation property by \cite[Theorem 3.7]{FriedlSchreveTillmann_ThurstonFox}. Moreover, each $G_i$ has a finite index locally indicable subgroup \cite[Flowchart 1, (C.19)]{AFW_3mfldBook} whose group ring is Hughes-free embeddable by \cref{thm:3mfld}. Thus, by \cref{thm:VCS_field}, each $k*G_i$ embeds into a division ring $\mathcal D_i$. Thus, if the JSJ decomposition only has one component, we are done.

    Now suppose that $n \geqslant 2$. Then each JSJ component has a non-empty boundary whose components are all either tori or Klein bottles. Then each fundamental group $\pi_1(M_i)$ is locally indicable (\cref{lem:locInd3mfld}), has the factorisation property (previous paragraph), and has a finite-index subgroup $H_i$ such that $k*H_i$ has a Hughes-free embedding (\cref{thm:3mfld}). Thus, each group algebra $k*\pi_1(M_i)$ has a Hughes-free embedding by \cref{lem:LinnellFIovergroup}. By \cref{cor:locIndGraphDivRing}, we conclude that $k*G$ embeds in a division ring.
    
    If $G$ is locally indicable, then the proof that the embedding into a division ring can be made Hughes-free is the same as the one given in the orientable case (see the proof of \cref{thm:graphMfld}). \qedhere
\end{proof}

We now conclude using the Prime Decomposition Theorem for non-orientable manifolds, due to Epstein \cite[Theorem 1.1]{Epstein_nonoriprime}.

\begin{thm}\label{thm:nonori3mfld}
    \cref{thm:3mfld,cor:3mfldKapl} hold without the assumption of orientability.
\end{thm}

\begin{proof}
    By the non-orientable Prime Decomposition Theorem (see \cite[Theorems 3.1 and 3.2]{Bonahon_3mfldsSurvey}, $M$ can be cut along finitely many essential, $2$-sided copies of $S^2$ or $\R P^2$, such that each of the components (after cutting) is irreducible or simply connected. By \cref{lem:noProj}, the assumption that $G$ is torsion-free implies that in fact there are no $\R P^2$'s and that $G$ decomposes as the free product of the fundamental groups of the irreducible pieces. We now conclude as in the proof of \cref{thm:3mfld}.
\end{proof}

\bibliographystyle{alpha}
\bibliography{bib}

\end{document}